\documentclass[a4paper, 11pt]{amsart}
\usepackage{amsmath, amsthm, amscd, amssymb, amsfonts, amsxtra, amssymb, latexsym}
\usepackage{enumerate}
\usepackage{verbatim}
\usepackage{float}
\usepackage{graphicx,epsfig,tikz}
\usepackage{pb-diagram}

\usepackage{array}
\usepackage{hyperref}
\hypersetup{colorlinks = true,	allcolors  = blue}

\newcommand{\G}{\Gamma}
\newcommand{\Z}{\mathbb{Z}}
\newcommand{\R}{\mathbb{R}}

\newcommand{\N}{\mathbb{N}}
\newcommand{\ff}{\mathbb{F}}
\newcommand{\lcm}{\operatorname{lcm}}

\newcommand{\sk}{\smallskip}
\newcommand{\msk}{\medskip}

\newtheorem{thm}{Theorem}[section]
\newtheorem{prop}[thm]{Proposition}
\newtheorem{lem}[thm]{Lemma}
\newtheorem{coro}[thm]{Corollary}

\theoremstyle{definition}
\newtheorem{rem}[thm]{Remark}
\newtheorem{exam}[thm]{Example}
\newtheorem{defi}[thm]{Definition}

\theoremstyle{remark}

\usepackage{color}

\begin{document}
\numberwithin{equation}{section}
\title[Integral equienergetic non-isospectral unitary Cayley graphs]{Integral equienergetic non-isospectral \\ unitary Cayley graphs}
\author[R.A.\@ Podest\'a, D.E.\@ Videla]{Ricardo A.\@ Podest\'a, Denis E.\@ Videla}
\dedicatory{\today}
\thanks{Partially supported by CONICET and SECyT-UNC}

\address{Ricardo A.\@ Podest\'a, FaMAF -- CIEM (CONICET), Universidad Nacional de C\'ordoba, \newline
	Av.\@ Medina Allende 2144, Ciudad Universitaria, (5000) C\'ordoba, Argentina. 
	\newline {\it E-mail: podesta@famaf.unc.edu.ar}}
\address{Denis E.\@ Videla, FaMAF -- CIEM (CONICET), Universidad Nacional de C\'ordoba, \newline
	Av.\@ Medina Allende 2144, Ciudad Universitaria,  (5000) C\'ordoba, Argentina. 
	\newline {\it E-mail: devidela@famaf.unc.edu.ar}}

\begin{abstract}
We prove that the Cayley graphs $X(G,S)$ and $X^+(G,S)$ are equienergetic for any abelian group $G$ and 
any symmetric subset $S$. We then focus on the family of 
unitary Cayley graphs $G_R=X(R,R^*)$, where $R$ is a finite commutative ring with identity.
We show that under mild conditions, $\{G_R, G_R^+\}$ are pairs of integral equienergetic non-isospectral graphs (generically connected and non-bipartite). Then, we obtain conditions such that 
$\{G_R, \bar G_R\}$ are equienergetic non-isospectral graphs.
Finally, we characterize all integral equienergetic non-isospectral triples $\{G_R, G_R^+, \bar G_R \}$ 
such that all the graphs are also Ramanujan. 
\end{abstract}

\maketitle

\section{Introduction}
This paper deals with the spectrum and the energy of Cayley graphs and Cayley sum graphs. 
Our main goal is to give a general construction of infinite pairs of integral equienergetic non-isospectral regular graphs with some nice extra properties like being connected, non-bipartite or Ramanujan (or all of them). 
We will focus on the family of unitary Cayley (sum) graphs over finite rings. The graphs will be undirected, but one of the graphs of the pairs can be taken either with or without loops.  

If $\Gamma$ is a graph of $n$ vertices, the eigenvalues of $\Gamma$ are the eigenvalues $\{\lambda_i\}_{i=1}^n$ of its adjacency matrix. The \textit{spectrum} of $\Gamma$
is the set of all the different eigenvalues $\{\lambda_{i_j}\}$ of $\Gamma$ counted with their multiplicities $\{e_{i_j}\}$,
and it is usually denoted by
$$Spec(\Gamma) = \{ [\lambda_{i_1}]^{e_1}, \ldots,[\lambda_{i_s}]^{e_{i_s}}\}$$ 
where $\lambda_{i_1} >\cdots > \lambda_{i_s}$. The spectrum is \textit{symmetric} if for every eigenvalue $\lambda$, its opposite $-\lambda$ is also an eigenvalue with the same multiplicity as $\lambda$. 
The graph $\Gamma$ is called \textit{integral} if $Spec(\G) \subset \Z$, i.e.\@ if all of its eigenvalues are integers. 
The \textit{energy} of $\Gamma$ is defined by 
$$E(\Gamma) = \textstyle \sum_{i=1}^n |\lambda_i| = \sum_{j=1}^s e_{i_j} |\lambda_{i_j}|.$$
We refer to the books \cite{BH} or \cite{CDS} for a complete viewpoint of spectral theory of graphs, and to \cite{Gu} for a survey on the energy of graphs.

\subsubsection*{Equienergetic non-isospectral graphs}
Let $\Gamma_1$ and $\Gamma_2$ be two graphs with the same number of vertices. The graphs are \textit{isospectral} 
(or \textit{cospectral}) if $Spec(\Gamma_1) = Spec(\Gamma_2)$ and \textit{equienergetic} if $E(\Gamma_1) = E(\Gamma_2)$.
It is clear by the definitions that isospectrality implies equienergeticity, but the converse is false in general. 
Thus, we are interested in the construction of equienergetic pairs of graphs which are non-isospectral. 
The smallest such pair is given by the 4-cycle $C_4$ and two disjoint copies of $K_2$, $K_2 \otimes K_2$ of 4-vertices  
or, if one wants connected graphs, by the $5$-cycle $C_5$ and the $5$-wheel 
$W_5=C_4 +K_1$ (the join of $C_4$ with and edge) of 5-vertices. In fact, we have 
$Spec(C_4) = \{[2]^2, [0]^2\}$ and $Spec(K_2 \otimes K_2) = \{[1]^2, [-1]^2 \}$ and also $Spec(C_5)=\{[2]^1, [\tfrac{\sqrt{5} -1}{2}]^2, [-\tfrac{\sqrt{5} +1}{2}]^2\}$ and $Spec(W_5) = \{ [\sqrt{5}+1]^1, [0]^2, [-\sqrt{5}+1]^1, [-2]^1\}$. Hence, we get $E(C_4)= E(K_2 \otimes K_2)=4$ and $E(C_5)=E(W_5)=2(\sqrt 5 +1)$. 
Notice that the first pair is integral while the second not. So, a pair of integral connected equienergetic non-isospectral graphs (without loops) must have at least 6 vertices. 

Although there are many examples of pairs of equienergetic non-isospectral graphs in the literature, there are few systematic constructions. We are only aware of the following four:
\begin{enumerate}[($a$)]
	\item \textit{Kronecker products}. In 2004, Balakrishnan (\cite{Ba}) showed that the graphs $\G \otimes (K_2 \otimes K_2)$ and 
	$\G \otimes C_4$ are equienergetic and non-isospectral, where $\Gamma$ is a non-trivial graph and $\otimes$ denotes the Kronecker product. Here, the first graph is not connected.
	
	\sk
	
	\item \textit{Iterated line graphs}. In the same year, Ramane et al (\cite{Ra}, 2004) proved that for two $k$-regular graphs $\Gamma_1$ and $\Gamma_2$ with the same number of vertices and $k \ge 3$,
	the iterated line graphs $L^r(\Gamma_1)$ and $L^r(\Gamma_2)$ are equienergetic for every $r\ge 2$.  Thus, if $\Gamma_1$ and $\Gamma_2$ are connected and non-isospectral, then $L^r(\Gamma_1)$ and $L^r(\Gamma_2)$ are connected equienergetic non-isospectral graphs.
	
	\sk
	
	\item \textit{Gcd-graphs}. Five years later Ilic (\cite{Il}, 2009) obtained families of $k$ hyperenergetic equienergetic non-isospectral gcd-graphs for any $k\in \N$. Namely, given $n=p_1 \cdots p_k$ with $p_1, \ldots, p_k$ primes, take the graphs $X_n(1,p_1), X_n(p_1,p_2), \ldots, X_n(p_{k-1},p_k)$ 
 	where $X_n(p_{j-1},p_j)$ have vertex set $\Z_n$ and edge set $E_j=\{ \{a,b\}:(a-b, n) \in \{p_{j-1},p_j\} \}$, where $p_0=1$.  
	
	\sk
	
	\item \textit{Doubles}. More recently, Ganie, Pirzada and Iv\'anyi (\cite{GPI}, 2014) constructed several pairs of 
	equienergetic non-isospectral graphs using the bipartite double $G^*$, the double $D[G]$ of $G$, and the iterated doubles $G^{k*}$, $D^k[G]$ of them. In particular, they proved that if $G$ is bipartite, $\{ G^*, D[G], G\otimes K_2 \}$ is a triple of equienergetic non-isospectral graphs. This suggests that we may restrict to the search of equienergetic non-isospectral pairs to non-bipartite graphs.
\end{enumerate}
Some of these methods can be combined to form new pairs of equienergetic non-isospectral graphs. For instance, in \cite{HX} it is proved that if $G_1$ and $G_2$ are $k$-regular graphs of the same order then $\{ L^2(G_1)^*, L^2(G_2)^* \}$, 
$\{ ({\overline{L^2(G_1)}})^*, ({\overline{L^2(G_2)}})^* \}$ and $\{ \overline{(L^2(G_1))^*}, \overline{(L^2(G_2))^*} \}$ are all pairs of bipartite equienergetic non-isospectral graphs.

Integral graphs were first considered by Harary and Schwenk in 1973 (\cite{HS}) when they posed the question 
\textit{``Which graphs have integral spectra?''} The problem seems to be very hard in great generality.
A survey of integral graphs from 2002 is \cite{BCRSS} focusing on trees, cubic graphs, 4-regular graphs and graphs of small size.
In \cite{AABS} the authors proved that only a small fraction of graphs of $n$-vertices are integral. More precisely, if $I(n)$ denotes the number of integral graphs of $n$-vertices then 
$I(n) \le 2^{\frac{n(n-1)}{2}-\frac{n}{400}}$,
where $2^{\frac{n(n-1)}{2}}$ is the number of graphs of $n$-vertices. Also, integral graphs may be of interest in the design of the network topology of perfect state transfer networks (see \cite{AABS} and references therein).

In this paper we present a new general construction to produce infinite pairs or triples of integral equienergetic non-isospectral regular graphs (typically connected and non-bipartite), using Cayley graphs, their complements and Cayley sum graphs, that we introduce next. 
After Abdollahi and Vatandoost asked \textit{``Which Cayley graphs are integral?''} in \cite{AV}, these graphs were studied further by Klotz and Sander (\cite{KS}, \cite{KS2}) and by Alperin and Peterson (\cite{AP}).

\subsubsection*{Cayley graphs}
Let $G$ be a finite abelian group and $S$ a subset of $G$ with $0\notin S$. The \textit{Cayley graph }$X(G,S)$ is the directed graph whose vertex set is $G$ and $v, w \in G$ form a directed edge of $\Gamma$ from $v$ to $w$ if $w-v \in S$. Since $0\notin S$ then $\Gamma$ has no loops. Analogously, the \textit{Cayley sum graph} $X^+(G,S)$ has the same vertex set $G$ but now $v,w\in G$ are connected in $\Gamma$ 
by an arrow from $v$ to $w$ if $v+w \in S$. Notice that if $S$ is symmetric,
that is $-S=S$, then 
$X(G,S)$ and $X^+(G,S)$ are $|S|$-regular undirected graphs. 
However, $X^+(G,S)$ may contain loops. In this case, there is a loop on vertex $x$ provided that  $x+x \in S$.
For an excellent survey of spectral properties of general Cayley graphs we refer the reader to \cite{LZ2}. 
By allowing loops, we will get several interesting new results. In particular, the smallest pair of integral equienergetic non-isospectral connected graphs is $\{ C_3, \hat P_3\}$, the 3-cycle and the 3-path with loops at the ends, with only 3 vertices each. 

One important special case of these graphs is obtained when $G$ is a finite ring with identity $R$ and $S$ is its group of units $R^*$. That is 
\begin{equation*} \label{GR} 
G_R = X(R,R^*) \qquad \text{and} \qquad 
G_R^+ = X^+(R,R^*),
\end{equation*} 
called the \textit{unitary Cayley graphs} and the \textit{unitary Cayley sum graphs}, respectively (the graphs $G_R^+$ are also known as \textit{closed unit graphs} and \textit{unit graphs}, if one does not allow loops).  
Unitary Cayley graphs were studied for instance in \cite{Ak+}, \cite{Il}, \cite{Ki+} and \cite{LZ}, and the unitary Cayley sum graphs were studied in the works of Maimani, Pournaki et al (see for instance \cite{AMP}, \cite{DMP}, \cite{MPY}) and recently in \cite{RAR}.
In this paper we will work with the graphs $G_R$ and $G_R^+$ as well as with the complements $\bar G_R$ of $G_R$.

\subsubsection*{Outline and results}
We now give the structure and summarize the main results of the paper.
In Section~2, we study the spectra of the graphs $X(G,S)$ and $X^+(G,S)$ for any abelian group $G$. 
In Proposition~\ref{Multiplicity} we compute the multiplicities of the eigenvalues 
of $X^+(G,S)$. In Theorem~\ref{equienergetic}, we show that the graphs $X(G,S)$ and $X^+(G,S)$ are equienergetic provided that $S$ is symmetric.
These graphs are generically non-isospectral. Under certain conditions on the characters of $G$, the spectrum of $X(G,S)$ determines that of $X^+(G,S)$ and the graphs $X(G,S)$ and $X^+(G,S)$ are in fact non-isospectral (see Proposition \ref{coromult}).

In the next section we consider unitary Cayley graph over rings $G_R$ and $G_R^+$,  
with $R$ a finite commutative ring with identity $1\ne 0$ and $R^*$ the group of units of $R$. 
The spectrum of $G_R$ is known (see \cite{Ki+}). By using this, we compute the spectrum of $G_R^+$ and show that $G_R$ and $G_R^+$ are integral equienergetic non-isospectral connected non-bipartite graphs, under certain conditions. In the case that $R$ is a local ring we only need that $|R|$ is odd (Proposition \ref{Spec GR+}). 
The general case is treated in Theorem \ref{XRR* equinoiso}, where we require that $|R|$ is of what we called odd-type (see Definition \ref{defi odd}).
The graph $G_R^+$ can be taken either with or without loops. As an application, in Proposition \ref{srg GR} we characterize all graphs $G_R$ and $G_R^+$ which are strongly regular graphs.

In Section 4 we consider the complementary graphs 
$\bar G_R = X(R,(R^*)^c \smallsetminus \{0\})$ of $G_R$. Both $G_R$ and $\bar G_R$ are loopless.
Using the known expressions for the energies of $G_R$ and $\bar G_R$ we obtain a general arithmetic condition on $R$ for $G_R$ to be equienergetic with its complement $\bar G_R$ (see \eqref{equien G bar G}). In Theorem~\ref{s12} we obtain explicit conditions when $R$ is local, a product of two local rings or a product of three finite fields, such that $E(G_R)=E(\bar G_R)$. As a consequence, in Corollary~\ref{teo crowns} we get that if $m$ is a prime power, then the complete $m$-multipartite graph $K_{m\times m}$ of $m^2$ vertices and the crown graph $H_{m,m}$ of $2m$ vertices are equienergetic and non-isospectral with their corresponding complements. In Corollary~\ref{ternas GG+Gbar} we produce infinitely many triples $\{G_R,G_R^+,\bar G_R\}$ of equienergetic non-isospectral graphs.

In Section 5, we deal with the construction of equienergetic non-isospectral pairs of graphs such that at least one of them is Ramanujan. 
We will use the known characterization of Ramanujan unitary Cayley graphs $G_R$ due to Liu and Zhou (\cite{LZ}, 2012). 
characterize all the pairs $\{G_R, G_R^+\}$ and $\{G_R,\bar G_R\}$ which are equienergetic non-isospectral, where at least one of the graphs is Ramanujan, distinguishing the cases when $R$ is a local ring or not (see Theorems \ref{teo ram gr} and \ref{teo ram gr nonlocal}). Then, we characterize all possible triples 
$\{G_R,\bar G_R, G_R^+ \}$ of equienergetic non-isospectral Ramanujan graphs (see Corollary \ref{coro gr ram} for $R$ local, Proposition \ref{prop tripla} for $R$ non-local and Corollary \ref{coro ternas} for $R=\Z_n$).

Finally, in the last section, by combining previous results we produce bigger sets of integral equienergetic non-isospectral graphs. For instance, in Example 6.6, we give a set of 23 integral equienergetic non-isospectral connected graphs.

\section{Equienergy of $X(G,S)$ and $X^+(G,S)$}
Here we compute the spectrum of $X^+(G,S)$ from that of $X(G,S)$ and show that $X(G,S)$ and $X^+(G,S)$ are equienergetic for $G$ abelian and $S$ symmetric, and that under certain conditions they are also non-isospectral graphs. 

Let $G$ a finite abelian group and $S$ a subset of $G$ not containing $0$. 
It is well-known that the spectra of $X(G,S)$ and $X^+(G,S)$ can be computed by using the irreducible 
characters $\widehat G$ of $G$. 
Given $\chi \in \widehat G$, i.e.\@ a group homomorphism $\chi: G \rightarrow \mathbb{S}^1 \subset \mathbb{C}^*$, one can define 
\begin{equation} \label{echis}
e_\chi = \chi(S) = \sum_{s\in S} \chi(s) \qquad \text{and} \qquad v_{\chi} = \big( \chi(g) \big)_{g\in G}.
\end{equation}
Note that $e_\chi \in \mathbb{C}$ and $v_\chi \in (\mathbb{S}^1)^n$ if $|G|=n$. 
We have the following well-known result.
\begin{lem} \label{X e Y}
In the previous notations:
\begin{enumerate}[$(a)$]
\item The eigenvalues of $X(G,S)$ are $\{ e_\chi \}_{\chi \in \widehat G}$ and $e_\chi$ has associated eigenvector $v_{\chi}$.

\item The eigenvalues of $X^+(G,S)$ are either $e_\chi= 0$ or $\{ \pm |e_\chi|\}_{\chi \in \widehat G} \subset \R$, such that:
	\begin{enumerate}[$(i)$]
			\item If $e_\chi =0$, their corresponding eigenvectors are $v_{\chi}$ and $v_{\chi^{-1}}$.
			
			\item If $e_\chi \ne 0$, the eigenvector associated to $\pm|e_\chi|$ is 
			$|e_\chi| v_{\chi} \pm e_\chi v_{\chi^{-1}}$.  
	\end{enumerate}
 \end{enumerate}
\end{lem}

Notice that $X(G,S)$ is integral if and only if $X^+(G,S)$ is integral. Also,
$|S|$ is always an eigenvalue of both $X(G,S)$ and $X^+(G,S)$. In fact, denoting by $\chi_0$ the principal character of $G$ (i.e.\@ $\chi_0(g)=1$ for every $g\in G$) then we have $e_{\chi_0}=|S|$. 
Moreover, $|S|$ is the principal eigenvalue of both graphs, i.e.\@ $\lambda_0=\lambda_0^+=|S|$ since $X(G,S)$ and $X^+(G,S)$ are $|S|$-regular. 
In general, if $\G$ is $k$-regular, then $k$ is an eigenvalue of $\G$ and $k \ge |\lambda|$ for an $\lambda \in Spec(\G)$, hence $k$ is called the \textit{trivial} or \textit{principal} eigenvalue of $\G$.

The lemma shows that $X(G,S)$ and $X^+(G,S)$ are generically non-isospectral. 
However, in some special cases these graphs could be isospectral or even the same graph. For instance, if $G=\ff_{2^m}$ or $G=\Z_{2^m}$ then $X(G,S)=X^+(G,S)$ for any subset $S$ of $G\smallsetminus \{0\}$  (see Lemma \ref{lema GRGR+}).

\subsubsection*{Equienergeticity}
Here we give simple conditions on $G$ and $S$ for $X(G,S)$ and $X^+(G,S)$ to be equienergetic.
We will need the following definition in the sequel.
\begin{defi}
	If $G$ is an abelian group and $S$ is a symmetric subset of $G$ with $0 \not \in S$, then we say that $(G,S)$ is an \textit{abelian symmetric pair}. 
\end{defi}

We begin by showing that there is a simple relation between the eigenvalue $e_\chi$ associated to a character $\chi$ of $G$ as in \eqref{echis} and the corresponding one associated to its inverse $\chi^{-1}$.
\begin{lem} \label{Ceros}
Let $S$ be a subset of a finite abelian group $G$ such that $0 \notin S$.   
Then $e_{\chi^{-1}} = \overline{e_\chi}$, and hence $e_\chi=0$ if and only if $e_{\chi^{-1}}=0$, for all $\chi \in \widehat{G}$.
Moreover, if $S$ is symmetric, then $e_\chi \in \R$ and $e_{\chi^{-1}} = e_\chi$ for all $\chi \in \widehat{G}$. 
\end{lem}

\begin{proof}
If $\chi\in \widehat{G}$, we have that
$$e_{\chi^{-1}} = \chi^{-1}(S) = \sum_{g\in S} \chi^{-1}(g) = \sum_{g \in S} \overline{\chi(g)} = \overline{\sum_{g \in S} \chi(g)} = \overline{\chi(S)} = \overline{e_\chi}.$$
and thus $e_\chi=0$ if and only if $e_{\chi^{-1}}=0$. 
Now, if $S$ is a symmetric set, then the adjacency matrix of $X(G,S)$ is symmetric. Thus, $e_\chi\in \mathbb{R}$ and hence 
$e_{\chi^{-1}} = \overline{e_\chi} = e_\chi$. 
\end{proof}

\begin{rem} \label{real chars}
The condition $\chi=\chi^{-1}$ is equivalent to $\chi$ being a real character of $G$, 
i.e.\@ $\chi(g) \in \mathbb{R}$ for all $g\in G$. For finite abelian groups, a real character  only takes values in 
$\mathbb{S}^1 \cap \R = \{\pm 1\}$. Also, if $|G|$ is odd then $\chi_0$ is the only real irreducible character of $G$. In fact, suppose $\chi$ is a nontrivial real character of $G$ 
and $\chi(g_0)=-1$ for some $g_0\in G$, then $1=\chi(e)=\chi(g_{0}^{|G|})=(-1)^{|G|}=-1$.
\end{rem}

An equivalence relation $\sim$ between irreducible characters of $G$ is given by
$\chi \sim \chi'$ if and only if $e_\chi = e_{\chi'}$
where $\chi, \chi' \in \widehat G$. We denote by $\widehat G /\!\!\sim$ the set of equivalence classes
\begin{equation} \label{clase ex}
[\chi] = \{\chi' \in \widehat{G}: \, \chi' \sim \chi \} = \{\chi' \in \widehat{G}: \, e_{\chi'} = e_{\chi}\}.
\end{equation}
We will also need to consider the following set of characters
\begin{equation} \label{clase -ex}
\widetilde{[\chi]} = \{\chi' \in \widehat{G}: \, e_{\chi'} = -e_{\chi}\}
\end{equation}
and the associated subsets 
\begin{align} \label{notations}
\begin{split}
& [\chi]_\R  = \{\rho \in [\chi] : \rho^{-1} = \rho \},  \qquad
[\chi]_{\R^c}  = \{\rho \in [\chi] :\rho^{-1} \ne \rho \}, \\[1mm]
& \widetilde{[\chi]}_\R = \{ \rho \in [\chi]_\R : e_\rho = -e_\chi \},   \quad
\widetilde{[\chi]}_{\R^c} = \{\rho \in [\chi]_{\R^c} : e_\rho = -e_\chi \}.  
\end{split}
\end{align}
Hence, 
$[\chi] = [\chi]_\R \cup [\chi]_{\R^c}$ and $\widetilde{[\chi]} = \widetilde{[\chi]}_\R \cup \widetilde{[\chi]}_{\R^c}$.

We denote by $m(\lambda)$ (resp.\@ $m^{+}(\lambda)$) the multiplicity of 
the eigenvalue $\lambda$ in $X(G,S)$ (resp.\@ $X^+(G,S)$), with the convention that $m(\lambda)=0$ (resp.\@ $m^+(\lambda)=0$) if $\lambda$ is not an eigenvalue of $X(G,S)$ (resp.\@ $X^+(G,S)$).
By the independence of the characters, it is clear that if $e_\chi \ne 0$ then 
\begin{equation} \label{multi echi}
m(e_\chi) = \# [\chi] \qquad \text{and} \qquad m(-e_\chi) = \# \widetilde{[\chi]}.
\end{equation}
We now compute the multiplicities of the eigenvalues for general Cayley sum graphs.

\begin{prop} \label{Multiplicity}
Let $(G,S)$ be a finite abelian symmetric pair. 
If $\chi \in \widehat G$, then the eigenvalues 
 of $X^+(G,S)$ are $\pm e_\chi \in \R$ with multiplicities given by
\begin{equation} \label{m+}
\begin{split}
m^{+}(e_\chi) = \#[\chi]_\R + \tfrac 12 \# [\chi]_{\R^c} + \tfrac 12 \#  \widetilde{[\chi]}_{\R^c},\\ 
m^{+}(-e_\chi) =  \#  \widetilde{[\chi]}_\R +  \tfrac 12 \#  [\chi]_{\R^c} + \tfrac 12 \# \widetilde{[\chi]}_{\R^c} .
\end{split}
\end{equation}
Moreover, the relation with the multiplicities of $X(G,S)$ is given by $m^{+}(0)=m(0)$ 
and  
\begin{equation} \label{rel m and m+}
m(e_\chi) + m(-e_\chi) = m^{+}(e_\chi)+m^{+}(-e_\chi)
\end{equation}
for $e_\chi\neq 0$.
\end{prop}

\begin{proof}
By Lemma \ref{X e Y}, the eigenvalues of $X^{+}(G,S)$ are either $e_\chi=0$ (with eigenvectors $v_\chi$ and $v_{\chi^{-1}}$) 
or $\pm|e_\chi|$ (with eigenvector $V_\chi^\pm = |e_\chi|v_\chi \pm e_\chi v_{\chi^{-1}})$ for $\chi \in \widehat G$. 
Moreover, since $S$ is symmetric the eigenvalues are real (the adjacency matrix of $X(G,S)$ is symmetric) and hence given by $\pm e_\chi$.
We will see that the only characters that can contribute to the multiplicity of $\pm e_\chi$ are those $\rho \in \widehat{G}$ such that either $\rho \in [\chi]$ or $\rho \in \widetilde{[\chi]}$. 
For clarity, we split the proof into cases.
	
\smallskip 
\noindent 
($a$) Suppose first that $\rho\in [\chi]_\R$, that is $\rho\sim \chi$ and $\rho^{-1} =\rho$. 
If $e_\rho> 0$, then $|e_\rho| = e_\rho$ is an eigenvalue with eigenvector $V_\rho^+ = 2 e_\rho v_\rho$, by Lemma~\ref{X e Y}. 
Moreover, $-|e_\rho|$ is not an eigenvalue because we would have $V_\rho^- = e_\rho v_{\rho} - e_\rho v_{\rho^{-1}} = 0$. 
Similarly, if $e_\rho<0$, then $|e_\rho|= -e_\rho$ is not an eigenvalue, although $-|e_\rho| = e_\rho$ is an eigenvalue with eigenvector 
$V_\rho^- = -2 e_\rho v_\rho$. Therefore, each real character $\rho \sim \chi$ contributes $1$ to the multiplicity of 
$e_\chi$ and $0$ to the one of $-e_\chi$.  
	
\smallskip
\noindent ($b$)
Now, assume that $\rho \in [\chi]_{\R^c}$, that is $\rho \sim \chi$ and $\rho^{-1} \ne \rho$. 
Firstly, if $e_\rho>0$, then $|e_\rho| = e_\rho$ is an eigenvalue with eigenvector 
$V_\rho^+ = e_\rho (v_\rho + v_{\rho^{-1}})$. Notice that, by Lemma \ref{Ceros}, $e_{\rho^{-1}} = e_\rho$, since $S$ is symmetric. Thus, 
$e_{\rho^{-1}}$ is an eigenvalue with eigenvector 
$$ V_{\rho^{-1}}^+ = |e_{\rho^{-1}}| v_{\rho^{-1}}+e_{\rho^{-1}} v_\rho = e_{\rho} (v_\rho + v_{\rho^{-1}}) = V_\rho^+.$$ 
Hence, $\rho^{-1}$ does not contribute to the multiplicity of $e_\chi$.
Therefore, the characters $\rho$ and $\rho^{-1}$ both contribute only one to the multiplicity of the eigenvalue $e_\chi$. 
On the other hand, we have that 
$-|e_\rho| = -e_{\rho} = -|e_{\rho^{-1}}|$ is an eigenvalue with eigenvectors 
$V_\rho^- = e_{\rho} (v_{\rho} - v_{\rho^{-1}})$ and $V_{\rho^{-1}}^- = e_{\rho}(v_{\rho^{-1}} - v_\rho)$. 
Then $V_{\rho^{-1}}^- = -V_\rho^-$ and, thus, the characters $\rho$ and $\rho^{-1}$ contribute only one to $m^+(-e_\rho)$.

Secondly, if $e_\rho<0$, then  $|e_\rho| = -e_\rho = |e_{\rho^{-1}}|$ is an eigenvalue of $X^{+}(G,S)$ with eigenvectors 
$V_\rho^+ = e_{\rho} (v_{\rho^{-1}} - v_\rho)$ and 
$V_{\rho^{-1}}^+ = e_{\rho^{-1}} (v_\rho - v_{\rho^{-1}})$,  
and thus $V_{\rho^{-1}}^+ = -V_{\rho}^+$. 
On the other hand,
$-|e_\rho| = e_{\rho} = -|e_{\rho^{-1}}|$ is an eigenvalue of $X^{+}(G,S)$ with eigenvector $-e_{\rho}(v_\rho + v_{\rho^{-1}})$, and therefore the characters $\rho$ and $\rho^{-1}$ contribute only one to the multiplicity of the eigenvalue $\pm e_\rho$.

\smallskip 
($c$) Suppose now that $\rho\in \widetilde{[\chi]}$, that is $\rho \in \widehat{G}$ and $e_{\rho} = -e_\chi$. 
By proceeding similarly as before we have that if $\rho$ is a real character this contributes in $1$ to the multiplicity of $-e_\chi$ and does not contribute to the multiplicity of $e_\chi$; and, on the other hand, if $\rho$ is a non-real character then $\rho$ and $\rho^{-1}$ contribute in $1$ to the multiplicity of the eigenvalues $\pm e_\chi$

By putting together the information in ($a$), ($b$) and ($c$) we get \eqref{m+}.

Now, if $e_\chi\ne 0$ then, by \eqref{notations} and \eqref{multi echi} we have
$$ m(e_\chi) + m(-e_\chi) = \#[\chi] + \# \widetilde{[\chi]} = (\# [\chi]_\R + \# [\chi]_{\R^c}) + (\# \widetilde{[\chi]}_\R + \# \widetilde{[\chi]}_{\R^c}).$$
Therefore, by \eqref{m+}, we have $m(e_\chi) +m(-e_\chi) = m^{+}(e_\chi) + m^{+}(-e_\chi)$, as it was to be shown.

Finally suppose that $\chi\in\widehat{G}$ with $e_\chi=0$. If $\chi^{-1} = \chi$, the contribution of $\chi$ to the multiplicity of $0$ is one, since $v_\chi = v_{\chi^{-1}}$. 
On the other hand, if $\chi^{-1}\ne \chi$, then $e_{\chi^{-1}}=0$ by Lemma \ref{Ceros}. In this case both $e_\chi$ and $e_{\chi^{-1}}$ are eigenvalues with eigenvectors $v_{\chi}$ and $v_{\chi^{-1}}$ and thus $m^{+}(0)=m(0)$.
\end{proof}

We now show spectrally that $X(G,S)$ and $X^+(G,S)$ share some structural properties. 
\begin{coro} \label{cxo nobip}
$X(G,S)$ is a connected and non-bipartite if and only if $X^+(G,S)$ is a connected and non-bipartite. 
\end{coro}

\begin{proof}
Recall that a $k$-regular graph is connected if and only if $m(\lambda_0)=1$ and it is non-bipartite if and only if $m(-\lambda_0)=0$, where $\lambda_0=k$ is the principal eigenvalue. For the graphs $X(G,S)$ and $X^+(G,S)$, since $\lambda_0=e_{\chi_0}=\lambda_0^+$, by \eqref{rel m and m+} we have 
$m(\lambda_0) + m(-\lambda_0) = m^{+}(\lambda_0) + m^{+}(-\lambda_0)$.
If $X(G,S)$ is connected and non-bipartite we have that $1=m^{+}(\lambda_0) + m^{+}(-\lambda_0)$ and since $m^{+}(\lambda_0) \ge 1$ we must have $m^{+}(\lambda_0)=1$ and $m^{+}(-\lambda_0)=0$. The converse is analogous.
\end{proof}

We are now in a position to show that Cayley graphs and Cayley sum graphs defined over the same abelian symmetric pair 
are always equienergetic.

\begin{thm} \label{equienergetic}
Let $(G,S)$ be a finite abelian symmetric pair. 
Then, the regular graphs $X(G,S)$ and $X^+(G,S)$ are equienergetic.
\end{thm}

\begin{proof}
Denote by $\widehat G /\!\!\approx$
the set of equivalence classes of the relation $\approx$ given by 
$\chi \approx \chi'$ if and only if $e_\chi = \pm e_{\chi'}$.
That is, $\widehat G /\!\!\approx$ equals $(\widehat{G} / \!\!\sim)/_{\{\pm 1\}}$.
We have
$$E(X(G,S)) = \sum_{\chi \in \widehat{G}} |e_\chi| = \sum_{\chi \in \widehat{G}/\approx} 
	\{ m(e_\chi) + m(-e_\chi) \} \,|e_\chi|. $$ 
Now, by \eqref{rel m and m+} in Proposition \ref{Multiplicity}	we get 
$$ E(X(G,S)) = 	\sum_{\chi\in \widehat{G}/\approx} \{ m^{+}(e_\chi) + m^{+}(-e_\chi) \} \, |e_\chi| = E(X^+(G,S)),$$
and the result follows.
\end{proof}

\begin{exam}[\textit{Circulant graphs}] \label{circulant}
Circulant graphs, introduced in \cite{Fu}, are Cayley graphs defined over cyclic groups, that is of the form $X(\Z_n, S)$.
If $S$ is a symmetric subset of $\Z_n$ not containing $0$, then the graphs  $X(\Z_n,S)$ and $X^+(\Z_n,S)$ are equienergetic. 
This includes the cases of $n$-cycles $C_n$ and $n$-paths 
with two loops at the ends $\hat P_n$, taking $S=\{\pm 1 \pmod n\}$, and the unitary Cayley graphs 
$U_n$, taking $S=\Z_n^*$.
Namely, we have 
$$C_n=X(\Z_n,\{\pm 1\}), \quad \hat P_{n} =X^+(\Z_{n},\{\pm 1\}), \quad \text{ and } \quad U_n = X(\Z_n,\Z_n^*),$$
where $n$ is odd for $\hat P_n$ (if $n>2$ is even then $\hat P_n=C_n$). 
Thus, we obtain that 
$E(C_n)=E(\hat P_n)$ and $E(U_n)=E(U_n^+)$, 
where $U_n^+ := X^+(\Z_n,\Z_n^*)$. 
Note that $U_3=C_3$ and $U_3^+ =\hat P_3$, but in general these graphs are all different for odd $n \ge 5$.
\hfill $\lozenge$
\end{exam}

\subsubsection*{Non-isospectrality}
We recall that one of our goals is to construct pairs of equienergetic non-isospectral graphs of the form $X(G,S)$ and 
$X^+(G,S)$.

We first give a simple condition for a pair $X(G,S)$ and 
$X^+(G,S)$ to be isospectral. 

\begin{coro} \label{isospcond}
Let $(G,S)$ be an abelian symmetric pair. 
If $Spec(X(G,S))$ and $Spec(X^+(G,S))$ are both symmetric, then $X(G,S)$ and $X^+(G,S)$ are isospectral. 	
\end{coro}

\begin{proof}
By hypothesis, $m(e_\chi) = m(-e_\chi)$ and $m^+(e_\chi) = m^+(-e_\chi)$ for every $\chi \in \hat G$. 
From \eqref{rel m and m+} in Proposition \ref{Multiplicity}, we immediately have that $m(e_\chi) = m^+(e_\chi)$ for every $\chi \in \hat G$, as desired. 
\end{proof}

As a consequence of Proposition \ref{Multiplicity}, we obtain the following condition for non-isospectrality.
\begin{prop} \label{coromult}
Let $(G,S)$ be a finite abelian symmetric pair such that the principal character $\chi_0$ is the only real character of $G$ 
$($for instance if $|G|$ is odd$)$. 
For each $\chi \in \widehat G \smallsetminus \{\chi_0\}$, 
with $e_\chi\neq 0, \pm |S|$, we have 
\begin{equation} \label{equal m+s}
m^{+}(e_\chi) = m^{+}(-e_\chi) = \tfrac{1}{2} \{ m(e_\chi) + m(-e_\chi) \}.
\end{equation} 
In particular, $Spec(X(G,S))$ determines $Spec(X^+(G,S))$.
If, in addition, $G$ has a non-trivial character $\chi$ such that $-e_{\chi}$ is not an eigenvalue of $X(G,S)$, then $X(G,S)$ and $X^{+}(G,S)$ are non-isospectral.
\end{prop}

\begin{proof}
By Remark \ref{real chars} and \eqref{notations} we have that $[\chi]_\R = \widetilde{[\chi]}_\R = \varnothing$. Hence, by 
Proposition~\ref{Multiplicity}, 
we get that  $m^{+}(e_\chi) = m^{+}(-e_\chi)$ and thus $2m^{+}(e_\chi) = m(e_\chi) + m(-e_\chi)$, from which \eqref{equal m+s} follows.

Now, if $\chi_0 \ne \chi \in \widehat G$ such that 
$-e_{\chi}$ is not an eigenvalue of $X(G,S)$, then 
$$0 = m(-e_{\chi})< \tfrac 12 m(e_{\chi}) = m^{+}(-e_{\chi}).$$ 
This implies the last assertion in the statement.
\end{proof}

\begin{exam} \label{c3p3}
($i$) The odd cycles and odd paths with loops at the ends are equienergetic non-isospectral 2-regular graphs. 
Indeed, $C_{2n+1}= X(\Z_{2n+1},\{\pm 1\})$ and $\hat P_{2n+1} = X^+(\Z_{2n+1},\{\pm 1\})$ are equienergetic by 
Example \ref{circulant}.
We now show that they have different spectra. 
The spectrum of $C_n$ is well-known, for $n\ge 1$ we have
$$Spec(C_{2n+1}) = \big\{ 2\cos(\tfrac{2\pi j}{2n+1}) \big\}_{0\le j \le 2n}.$$ 
If $\omega= e^{\frac{2\pi i}{2n+1}}$ denotes the $(2n+1)$-th primitive root of unity and $\chi$ is the associated character, then $-e_\chi = -2 Re (\omega) \not \in Spec(C_{2n+1})$ and hence, by Proposition \ref{coromult}, the graphs $C_{2n+1}$ and $\hat P_{2n+1}$ are not isospectral.

Note that $Spec(C_{2n+1})$ is integral if and only if $n=1$. In this case we have 
$$Spec(C_3)=\{ [2]^1,[-1]^{2} \} \qquad \text{and} \qquad Spec(\hat P_3) = \{ [2]^1, [1]^1,[-1]^1 \},$$ 
since $\hat P_3$ has adjacency matrix $A=\left[\begin{smallmatrix}
1&1&0 \\ 1&0&1 \\ 0&1&1 \end{smallmatrix}\right]$, and hence $E(C_3)=E(\hat P_3)=4$. 
In this way, the connected 2-regular graphs $C_3$ and $\hat P_3$ 
$$
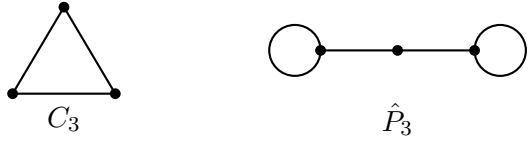
\begin{tikzpicture}[scale=.675, thick]
\fill (-1,0) circle (3pt); 
\fill (1,0) circle (3pt); 
\fill (0,1.7) circle (3pt); 
\draw (-1,0) -- (1,0) -- (0,1.7) -- (-1,0);
\node at (0,-.5) {$C_3$};
\fill (5,0.85) circle (3pt); 
\fill (6.5,0.85) circle (3pt); 
\fill (8,0.85) circle (3pt); 
\draw (4.5,0.85) circle (0.5);
\draw (8.5,0.85) circle (0.5);
\draw (5,0.85) -- (6.5,0.85) -- (8,0.85);
\node at (6.5,-.5) {$\hat P_3$};
\end{tikzpicture} 
$$
are integral equienergetic non-isospectral, one having a cycle and no loops while the other is acyclic with loops. This is the smallest possible example of this kind.

($ii$) More generally, consider the subset $S_r=\{ \pm r \pmod n\}$ of $\Z_n$. If $(r,n)=1$ then $X(\Z_n,S_r)$ is isomorphic to $X(\Z_n, S_1)$ and  $X^+(\Z_n,S_r)$ is isomorphic to $X^+(\Z_n, S_1)$, since the $\mathbb{Z}_n$-automorphism which sends $1$ to $r$ maps $S_1$ in $S_r$. Then, all the graphs in the family
$\{ X(\Z_n,S_r), X^+(\Z_n, S_r)\}_{(r,n)=1}$ 
are equienergetic and each pair $\{ X(\Z_n,S_r), X^+(\Z_n, S_t)\}$, with $r,t$ coprime with $n$, is non-isospectral.
\hfill $\lozenge$
\end{exam}

\begin{rem}
	Note that under the hypothesis of Proposition \ref{coromult}, $Spec(X^+(G,S))$ does not determine $Spec(X(G,S))$. However, if in addition one has that $m(-e_\chi)=0$ or $m(-e_\chi)=m(e_\chi)$ for every $\chi\in \widehat G$, then $Spec(X^+(G,S))$ determines $Spec(X(G,S))$.
\end{rem}

Since the principal character $\chi_0 \in$ in $\widehat G$ gives rise to the principal eigenvalue in $\Gamma= X(G,S)$ from now on we will denote the eigenvalues of $\Gamma$ by $\lambda_0, \ldots, \lambda_{n-1}$ instead of $\lambda_1, \ldots,\lambda_n$, in order to have $\lambda_0 = e_{\chi_0}$.

In the sequel, we will use the following concept.
\begin{defi}
	If $\Gamma$ is a graph, we will say that $Spec(\Gamma)$ or $\G$ is \textit{almost symmetric} if the multiplicities of an eigenvalue $\lambda$ and of its opposite $-\lambda$ are the same, except for $\lambda_0$, i.e.\@ 
	$m(\lambda)=m(-\lambda)$ for every $\lambda \ne \lambda_0$ (note that $\lambda=0$ automatically satisfies this and it is allowed that $m(-\lambda)=0$). If in addition $|m(\lambda_0)-m(-\lambda_0)|=1$ holds, then we say that $Spec(\Gamma)$ or $\G$ is \textit{strongly almost symmetric}.
\end{defi}

It is known that the sum of the eigenvalues of non-directed graphs equals $0$. 
This is not true for the sum graphs $\G=X^+(G,S)$, in general. However, if $\G$ is almost symmetric, 
we have 
$\sum_{i} \lambda_i = |m(\lambda_0) - m(-\lambda_0)| \, \lambda_0$.
If, further, $\G$ is strongly almost symmetric then 
$\sum_{i} \lambda_i = \lambda_0$.

For strongly almost symmetric graphs, connectivity is equivalent to non-bipartiteness. 
\begin{prop} \label{conn=nobip}
	If $\G$ is a strongly almost symmetric graph, then $\G$ is connected if and only if $\G$ is non-bipartite.
\end{prop}

\begin{proof}
Since $\G$ is strongly almost symmetric, we have $m(\lambda_0)=t+1$ and $m(-\lambda_0)=t$ for some $t\in \N_0$, where $\lambda_0$ is the principal eigenvalue of $\G$. 
If $\G$ is connected, then $t=0$ and hence $-\lambda_0$ is not an eigenvalue of $\G$, thus $\G$ is non-bipartite. 
Conversely, if $\G$ is non-bipartite, $t=0$ and $m(\lambda_0)=1$, thus $\G$ is connected.  
\end{proof}

\section{Unitary Cayley graphs over finite rings}
Here we will produce pairs of integral equienergetic non-isospectral pairs of unitary Cayley (sum) graphs over rings.
The graphs will result connected and generically non-bipartite.

Let $R$ be a finite commutative ring with identity $1\ne 0$, $R^*$ be its group of units, and $S \subset R$ with $0 \notin S$.
Clearly, if the characteristic of $R$ is $2$, the graphs $X(R,S)$ and $X^+(R,S)$ are the same, so in the sequel we will assume that $\mathrm{char}(R) \ne 2$.
Also, notice that in odd characteristic, since $2$ is always a unit, multiplication by $2$ is a bijection of $R$. Thus, for any $y\in S$ we have that $y=2x \in S$ for some unique $x \in R$ and, hence, there are loops in exactly $|S|$ vertices of $X^+(R,S)$.

From now on, we will consider $S=R^*$ and the graphs 
$$G_R=X(R,R^*) \qquad \text{and} \qquad G_R^+=X^+(R,R^*).$$ 
Note that $(R,R^*)$ is an abelian pair since $u\in R^*$ if and only if $-u \in R^*$.
By the well-known Artin's structure theorem  
we have that 
\begin{equation} \label{artin desc}
R = R_1 \times \cdots \times R_s
\end{equation} 
where each $R_i$ is a local ring, that is having a unique maximal ideal $\frak{m}_i$. From now on, we put $r_i:=|R_i|$ and $m_i:=|\frak{m}_i|$ for $i=1,\ldots,s$.
Moreover, one also has the decomposition $R^* = R_1^* \times \cdots \times R_s^*$. 
This implies that 
\begin{equation} \label{artin units}
G_R = G_{R_1} \otimes \cdots \otimes G_{R_s} \qquad \text{and} \qquad G_R^+ = G_{R_1}^+ \otimes \cdots \otimes G_{R_s}^+
\end{equation} 
with $G_{R_i} = X(R_i,R_i^*)$, $G_{R_i}^+ = X^+(R_i,R_i^*)$, and where $\otimes$ denotes the Kronecker product. 

Every local ring has order $p^{m}$ and characteristic $p^{r}$ for some prime $p$ and $r,m \in \N$. Then, the conditions `$\mathrm{char}(R)$ is odd' and `$|R|$ is odd' are equivalent. In  fact, since $|R|=|R_1|\cdots |R_s|$ we have 
$\mathrm{char}(R) = \lcm_{1\le i \le s} \{\mathrm{char}(R_{i})\}$. 
Thus, the condition $|R|$ odd implies that $2 \in R^*$, and it can be shown that the converse also holds. Hence, we have the equivalence 
\begin{equation} \label{Rodd2R*}
|R| \text{ is odd } \quad \Leftrightarrow  \quad \mathrm{char}(R) \text{ is odd} \quad \Leftrightarrow\quad 2\in R^*.
\end{equation}

The spectrum of $G_R = X(R,R^*)$ is known. If $R$ is as in \eqref{artin desc}, put  
\begin{equation} \label{lambdaC} 
\lambda_{C} = (-1)^{|C|} \frac{|R^*|}{\prod\limits_{j \in C} (|R_j^*|/m_j)} = (-1)^{|C|} \prod\limits_{j \in C} m_j \prod\limits_{i \not\in C} |R_i^*|
\end{equation} 
for each subset $C \subseteq \{ 1, \ldots, s\}$. 
From \cite{Ki+} (see also \cite{LZ}), the eigenvalues of $G_R$ are 
\begin{equation} \label{spec GR}
\lambda = \begin{cases}
\lambda_C, & \quad \text{repeated $\prod\limits_{j \in C} (|R_j^*|/m_j)$ times,} \\ 
0,		 & \quad \text{with multiplicity $|R|- \prod\limits_{i=1}^s (1+ \tfrac{|R_i^*|}{m_i})$,}
\end{cases}
\end{equation}
where $C$ above runs over all the subsets of $\{1,2,\ldots,s\}$. 
Note that, a priori, different subsets $C$ can give the same eigenvalue.

\subsubsection*{Local rings} 
We first consider the case of a finite commutative local ring $(R, \frak m)$. Let $r=|R|$ and $m=|\frak m|$. 
Note that 
\begin{equation} \label{GRlocal}
G_R \simeq \begin{cases} K_r  & \qquad \text{if $R$ is a field}, \\[1mm] 
K_{\frac{r}{m}\times m} 		& \qquad \text{if $R$ is not a field},
\end{cases}
\end{equation}
where $K_r$ is the complete graph (this is obvious) and $K_{\frac rm\times m}$ is the complete $\frac rm$-multipartite graph of $\frac rm$ parts of size $m$ (see Proposition 2.2 in \cite{Ak+}).

We remark that in even characteristic $G_R$ coincides with $G_R^+$, as was recently observed in \cite{RAR} without proof.
For completeness, we prove 
this fact by showing that both graphs are isomorphic to a complete multipartite graph. 

\begin{lem} \label{lema GRGR+}
	If $R$ is a finite local ring with $|R|$ even then $G_R = G_R^+$. 
	In particular, $G_R^+$ is loopless.
\end{lem} 

\begin{proof}
Let $\frak m$ be the maximal ideal of $R$.
Notice that $x$ and $y$ are adjacent in $G_R^+$ if and only if $x+y\not \in \frak{m}$.
Since $2\in \frak m$ by \eqref{Rodd2R*}, we have that 
$x+y\in \frak{m}$ if and only if $x-y\in\frak m$. For instance, $x-y \in \frak m$ implies $(x-y)+2y=x+y \in \frak m$. Therefore, $G_R$ and $G_R^+$ have the same edges and thus $G_R=G_R^+$. 
\end{proof}

We now give the spectrum of $G_R^+$ in the case $|R|$ is odd and show that $\{G_R,G_R^+\}$ is an integral equienergetic non-isospectral pair of graphs in this case. 

\begin{prop} \label{Spec GR+}
Let $(R,\frak m)$ be a finite local ring with $|R|=r$ and $|\frak m|=m$.  
If $r$ is odd then $G_R^+$ has loops and its spectrum is strongly almost symmetric given by
	\begin{equation} \label{spec gr+} 
		Spec(G^+_R) =  \{ [r-m]^1, [m]^{\frac{r-m}{2m}}, [0]^{\frac rm (m-1)}, [-m]^{\frac{r-m}{2m}} \}
	\end{equation}
if $R$ is not a field $(m>1)$ and by
	\begin{equation} \label{spec gr+ f} 
		Spec(G^+_R) =  \{ [r-1]^1, [1]^{\frac{r-1}{2}}, [-1]^{\frac{r-1}{2}} \}
	\end{equation}
if $R$ is a field $(m=1)$. 
Moreover, $G_R$ and $G_R^+$ are integral equienergetic non-isospectral connected non-bipartite graphs. 
\end{prop}

\begin{proof}
Since $R$ is local, we have $R^* = R \smallsetminus \frak m$ and $s=1$ in the decompositions \eqref{artin desc} and 
\eqref{artin units}. Thus, by \eqref{lambdaC} and \eqref{spec GR}, the eigenvalues of $G_R$ are given by
\begin{itemize}
	\item $\lambda_{\varnothing}= |R^*|$ with multiplicity $1$, \msk 
	
	\item $\lambda_{\{1\}} = - \frac{|R^*|}{|R^*|/m}$ with multiplicity $\frac{|R^*|}{m}$ and \msk
	
	\item $0$ with multiplicity $r-(1+ \frac{r-m}{m})$, where $\frac rm \ne 1$ since $\frak m$ is a proper 
	ideal of $R$.
\end{itemize}
  
Hence, we have 
\begin{equation} \label{spec GR local}
Spec(G_R) = \begin{cases}
\{ [r-m]^1, [0]^{\frac rm(m-1)}, [-m]^{\frac rm -1}\} & \qquad \text{if $R$ is not a field } (m>1), \\[2mm]
\{ [r-1]^1, [-1]^{r-1}\} & \qquad \text{if $R$ is a field } (m=1). 
\end{cases}
\end{equation}
This implies that there is some $\chi \in \widehat R$ such that $e_\chi=-m$ and hence 
$$-e_\chi=m\not \in Spec(G_R),$$ 
since $r-m=m$ would imply $r=2m$, but $r$ is odd.
Thus, by \eqref{equal m+s} in Proposition \ref{coromult} we obtain \eqref{spec gr+} and \eqref{spec gr+ f} as desired, from which it follows that the spectrum of $G_R^+$ is strongly almost symmetric.

Finally, the graphs $G_R$ and $G_R^+$ are equienergetic by Theorem \ref{equienergetic}, since $(R,R^*)$ is an abelian symmetric pair, and integral and non-isospectral by \eqref{spec gr+}--\eqref{spec GR local}.
The graph $G_R$ is connected and non-bipartite by \eqref{spec GR local}, since $r\ne 2m$ by hypothesis. 
That the graph $G_R^+$ is connected and non-bipartite follows directly by \eqref{spec gr+} and \eqref{spec gr+ f}, or else by Corollary \ref{cxo nobip}.
\end{proof}

We now illustrate the previous result in the particular case of Galois rings.

\begin{exam}[\textit{Galois rings}] \label{GRings}
Let $p$ be an odd prime and $s, t \in \N$. 
Let $R=GR(p^s,t)$ be the finite Galois ring of $p^{st}$ elements. The ring $R$ is local with maximal ideal $\frak m=(p)$ and hence $m=p^{(s-1)t}$ and $|R^*|=p^{(s-1)t}(p^t-1)$. 
In particular, if $t=1$ then $R$ is the local ring $\Z_{p^s}$ while if $s=1$ we have that $R$ is the finite field $\ff_{p^t}$. Thus,  
\begin{equation*}
\begin{split}
Spec(G_R)   & = \{ [p^{(s-1)t}(p^t-1)]^{1}, [0]^{p^t(p^{(s-1)t}-1)}, [-p^{(s-1)t}]^{p^t-1} \}, \\
Spec(G_{R}^+) & = \{ [p^{(s-1)t}(p^t-1)]^{1}, [p^{(s-1)t}]^{\frac{p^t-1}{2}}, [0]^{p^t(p^{(s-1)t}-1)}, [-p^{(s-1)t}]^{\frac{p^t-1}{2}} \}.
\end{split}
\end{equation*}

Note that in the case that $R$ is a field $0$ is not an eigenvalue since $s=1$.  
If $R=\Z_{p^s}$ we have that $G_{p^s}=X(\Z_{p^s},\Z_{p^s}^*)$ is the complete $p$-partite graph and 
\begin{equation*} \label{spec Zpr} 
\begin{split}	
Spec(G_{p^s}) 	& = \{ [p^s-p^{s-1}]^1, [0]^{p^s-p}, [-p^{s-1}]^{p-1} \}, \\
Spec(G_{p^s}^+) & = \{[p^s-p^{s-1}]^1, [p^{s-1}]^{\frac{p-1}2}, [0]^{p^s-p},  [-p^{s-1}]^{\frac{p-1}2} \}.
\end{split}
\end{equation*}

In all the cases we have that $G_R$ and $G_R^+$ are integral equienergetic non-isospectral connected non-bipartite graphs 
and that $G_R^+$ is strongly almost symmetric. 
\hfill $\lozenge$
\end{exam}

\subsubsection*{Non-local rings} 
We will need the following notations. Given the decomposition \eqref{artin desc} of $R$, 
let $I_s(R) = \{1,2,\ldots,s\}$ and put
	$$I_e(R) := \{i \in I_s(R) : |R_i| \text{ is even}\} \qquad \text{and} \qquad 
	I_{o}(R) := \{i \in I_s(R) : |R_i| \text{ is odd}\}.$$ 
	For simplicity, we will write $I_e$ and $I_o$ when the ring $R$ is understood. We denote by $E(R)$ and $O(R)$ the \textit{even part} and \textit{odd part} of $R$ respectively, where 
	$$E(R) = \prod_{i\in I_{e}} R_{i} \qquad \text{and} \qquad O(R) = \prod_{j \in I_{o}} R_j .$$ 
	Clearly,
	$R \simeq E(R)\times O(R)$.

 We now give some basic structural results on $G_R$ in terms of the even and odd part of $R$.
	
	\begin{lem}\label{lem nonsym}
		Let $R\simeq R_1\times\cdots\times R_s$ be a finite commutative ring with identity with $R_i$ local for $i=1,\ldots,s$. Then we have:
		\begin{enumerate}[$(a)$]
			\item If $O(R)= \{0\}$, then $G_R = G_R^+$. \msk
			
			\item $E(R) \ne \{0\}$ if and only if $G_R^+$ is loopless. \msk
			
			\item If $r_i = 2m_i$ 
			for some $i\in I_{e}$, then $G_R$ and $G_R^+$ are isospectral bipartite graphs. \msk			
			
			\item $G_R$ is non-bipartite if and only if $2m_i<r_i$ for every $i=1,\ldots,s$.
		\end{enumerate} 
	\end{lem}

	\begin{proof}
$(a)$ If $O(R)=\{0\}$, then every $R_i$ has order a power of 2 and hence  
$G_{R_i} = G_{R_i}^+$ for $i=1,\ldots,s$, by Lemma \ref{lema GRGR+}.
The assertion thus follows directly from \eqref{artin units}.
		
\noindent $(b)$ If $E(R)=\{0\}$, then $|R|$ is odd and then $G_{R}^+$ has loops.
Otherwise, if $E(R)\ne \{0\}$, then $G_E^+ \simeq G_E$ and $G_E$ is loopless.
Since the Kronecker product of a loopless graph with a graph (with or without loops) is loopless, 
we obtain that of $G_R^+\simeq G_E\otimes G_O^{+}$ is loopless, as asserted.

\noindent $(c)$ Suppose that $r_i=2m_i$ for some $i\in I_{e}$. 
Then, $G_{R}$ and $G_{R}^+$ are bipartite graphs, since they are both  Kronecker products with a bipartite factor $G_{R_i}$, 
and therefore $G_R$ and $G_R^+$ have symmetric spectra. The result follows from Corollary \ref{isospcond}.

\noindent $(d)$	Clearly $(c)$ implies that if $G_R$ is non-bipartite then $2m_i<r_i$ for every $i=1,\ldots,s$.
Now assume that $2m_i<r_i$ for every $i=1,\ldots,s$. By \eqref{spec GR local} we have that $G_{R_i}$ is non-bipartite for all $i=1,\ldots,s$.
Thus, $-|R^*|$ cannot be an eigenvalue of $G_R$ since all of its eigenvalues are product of the eigenvalues of the  $G_{R_i}$'s, and hence $G_R$ is non-bipartite, as desired.  
\end{proof}
	
So, for instance, $G_{\Z_4 \times \ff_3}^+$ is a simple bipartite graph while $G_{\ff_4 \times \ff_3}^+$ is simple and non-bipartite.
	
\begin{defi} \label{defi odd}
In previous notations, we will say that a finite commutative ring with identity $R$ is of \textit{odd type} if $m_i < \frac{r_i}2$ for all $i\in I_{e}$ and $O(R)\ne \{0\}$. 
\end{defi}

Notice that if $\mathrm{char}(R)$ is odd, then $R$ is of odd type. Moreover if $R$ is a ring of odd type, then $\mathrm{char}(R)$ is not a power of $2$. Furthermore, for instance, it cannot have factors of the form $\Z_{2^n}$ for $n\in \N$.
For a local ring $R$, to be of odd type and to have odd cardinality are clearly equivalent conditions.

We now give the number of edges $e(G_R)$ and $e(G_R^+)$ of $G_R$ and $G_R^+$, respectively.
\begin{lem} \label{edges}
Let $R$ be a finite commutative ring and put $r=|R|$ and $k=|R^*|$. Then, we have
$e(G_R) = \tfrac 12{kr}$ and $e(G_R^+) = k[\frac{r+1}{2}]$. Also,   
$e(G_R) = e(G_R^+)$ if and only if $r$ is even. 
\end{lem}

\begin{proof}
Since $G_R$ is $k$-regular of $r$-vertices we have that $e(G_E)=\frac 12 kr$. By Lemma \ref{lem nonsym}, $G_R^+$ has loops if and only if $r$ is odd. If $r$ is even, then $G_R^+$ has no loops and hence $e(G_R^+)=e(G_R)$. If $r$ is odd, then $G_R^+$ has $k$ loops and the number of edges is 
$$e(G_R^+) = k + \tfrac 12 {k(k-1)} + \tfrac 12 {k(r-k)} = k+ \tfrac 12 {k(r-1)} = \tfrac 12{k(r+1)}.$$
Thus, $e(G_R^+)$ is as stated. The remaining assertion is straightforward.  
\end{proof}

We now give a general family of pairs of equienergetic and non-isospectral graphs.
This result generalizes Proposition \ref{Spec GR+}.

\begin{thm} \label{XRR* equinoiso}
Let $R$ be a ring of odd type.  
Then, $G_R= X(R,R^*)$ and $G_R^+= X^+(R,R^*)$ are integral equienergetic non-isospectral connected non-bipartite graphs. 
Moreover, $G_R^+$ is simple if $|R|$ is even while $G_R^+$ is strongly almost symmetric with loops if $|R|$ is odd. 
\end{thm}

\begin{proof}
By Theorem \ref{equienergetic}, the graphs $G_R$ and $G_R^+$ are equienergetic.
We will next show that they are non-isospectral. 

Assume first that $|R|$ is odd. Hence, the only real character of $R$ is the trivial one (Remark~\ref{real chars}). Thus, by Proposition \ref{coromult}, to show that 
$G_R$ and $G_R^+$ are non-isospectral it is enough to show that there is some non-trivial irreducible 
character $\chi$ of $R$ such that $-e_\chi$ is not an eigenvalue of $G_R$.

Let $R = R_1 \times \cdots \times R_s$ be the Artin's decomposition of $R$ in local rings as in \eqref{artin desc}, 
with $|R_i|=r_i$ and $|\frak{m}_i| = m_i$ for $i=1,\ldots,s$.

Suppose first that $R$ is reduced, that is $R$ has no non-trivial nilpotent elements. In this case, the rings $R_i$ 
are all fields (hence $m_i=1$) for all $i$. 
In this way, by \eqref{lambdaC}, the eigenvalues of $G_R$ have the form 
\begin{equation*} \label{eig C}
\lambda_{C}= (-1)^{|C|} \prod_{j \notin C} (r_j-1),
\end{equation*}
where $C \subseteq \{1,\ldots,s\}$. 
Notice that if $C =\{1,\ldots,s\}$, then 
$$\lambda_{C} = (-1)^{s}$$ 
is an eigenvalue of $G_R$ while $-\lambda_{C} = (-1)^{s+1}$ is not. In fact, $r_i > 2$ for every $i=1,\ldots,s$, 
since  $|R|$ is odd. 
This implies that there is some non-trivial character 
$\chi \in \widehat R$ such that 
$$-e_\chi \notin Spec(G_R).$$ 
By Proposition \ref{coromult}, we have that $G_R$ and $G_R^+$ are non-isospectral.

Now, suppose that none of the rings $R_i$ are fields, hence $m_i >1$ for every $i=1,\ldots,s$. Then,
the eigenvalues of $R$ are $0$ and $\lambda_C$ as given in \eqref{lambdaC}
and \eqref{spec GR}, that in this case we denote by $\mu_D$ with $D\subseteq \{1,\ldots,s\}$. 
Taking $D=\{1,\ldots,s\}$ we have that 
\begin{equation*} \label{eig D}
\mu_{D} = (-1)^{s} 
m_1 \cdots m_s
\end{equation*}
is an eigenvalue of $G_R$ while $-\mu_D$ is not. 
Proceeding as before, this implies that $G_R$ and $G_R^+$ are non-isospectral.
	
In the general case, we can always write $R = F \times L$ 
where $F=F_{1}\times\cdots\times F_{s}$ is a reduced ring with $F_i$ a field for all $i=1,\ldots,s$ and 
$L=L_{1} \times \cdots \times L_{t}$ where each $L_{i}$ is a local ring which is not a field, for all $j=1,\ldots,t$. 
Hence, $R^*=F^*\times L^*$ and 
$$G_R = G_F \otimes G_L.$$ 
If $C=\{1,\ldots,s\}$ and $D=\{1,\ldots,t\}$, we can take the eigenvalues $\lambda_{C}$ of $G_F$ and $\mu_{D}$ of $G_L$ as before. 
Since it is known that the eigenvalues of the Kronecker product is the product of the eigenvalues of its factors, we have that 
$$\lambda_{C} \mu_{D} = (-1)^{s+t} m_1 \cdots m_t$$ 
is an eigenvalue of $G_R$ while $-\lambda_{C} \mu_{D}$ is not. 
Therefore, $G_R$ and $G_R^+$ are non-isospectral.

Now assume we are in the general case, i.e.\@ $R$ is a commutative ring of odd type and let $E=E(R)$ and $O=O(R)$. 
By Lemma \ref{lem nonsym}, we have that  $G_E\simeq G_E^+$, 
and in particular $G_E$ and $G_E^+$ have the same spectra. 
On the other hand, 
$O(R)$ is a nontrivial ring of odd cardinality.  
By putting $O(R)=F \times L$ and defining $C,D$ as above, we can take 
$\lambda_C$ and $\mu_D$ associated to $C$ and $D$, 
so that $\lambda_C\mu_D$ is an eigenvalue of $O(R)$ but $-\lambda_C\mu_D$ is not. 

Recall that the eigenvalues of $G_E$ are given by 
$$\eta_W= (-1)^{|W|} \prod_{i\in W} m_i \prod_{j\in I_e\smallsetminus W}|R_{j}^*|,$$ 
where $W\subseteq I_e$.
Since $2m_i< r_i$ for all $i\in I_{e}$, 
then $m_i < |R_{i}^*|$ for all $i\in I_{e}$. 
Thus if $W_1\subseteq W_2$ are different subsets of $I_e$ then $|\eta_{W_1}|> |\eta_{W_2}|$.
In particular, if $W$ is a proper subset of $I_e$ then $|\eta_W|>|\rho_E|$ where 
$\rho_E:= \eta_{I_e}$.
This implies that 
$$\rho_E=(-1)^{|I_{e}|} \prod_{i\in I_{e}}m_i \in Spec(G_E) \qquad \text{and} \qquad 
-\rho_E\not \in Spec(G_E) .$$
On the other hand, notice that the decomposition $R=E\times O$  induces the decompositions $G_R\simeq G_E\otimes G_O$ and
$G_R^+\simeq G_E^+ \otimes G_O^+$, by \eqref{artin units}. Since $G_E\simeq G_E^+$, we obtain
$$G_R^+ \simeq G_E \otimes G_O ^+.$$
So, we have that $-\rho_E\lambda_C \mu_D\in Spec(G_R^+)$ and $-\rho_E\lambda_C \mu_D \not \in Spec(G_R)$. 
Therefore $G_R$ and $G_R^+$ are non-isospectral, as desired.

We now prove that $G_R$ and $G_R^+$ are both connected and non-bipartite. By \eqref{artin units}, the eigenvalues of $G_R$ are all the products of the eigenvalues of the $G_{R_i}$'s. Since each $G_{R_i}$ is connected by Proposition \ref{Spec GR+}, then $G_R$ is connected. Also, $G_R$ is non-bipartite since $R$ is of odd type, by ($d$) in Lemma \ref{lem nonsym}. By Corollary \eqref{cxo nobip}, $G_R^+$ is also connected and non-bipartite.

Finally, the integrality of $Spec(G_R)$ is known and this clearly implies the integrality of $Spec(G_R^+)$.
The last assertion follows directly from Proposition \ref{coromult} and $(c)$ of Lemma \ref{lem nonsym}.
\end{proof}

\begin{exam}
If $R=\ff_{2^n} \times \ff_q$, with $n\ge 2$ (see Lemma \ref{lem nonsym}) and $q$ odd, then $\{G_R,G_R^+\}$ is a pair of integral equienergetic non-isospectral graphs without loops having the same number of edges. The smallest such graph is $4$-regular with 12 vertices, corresponding to $R=\ff_4 \times \ff_3$. \hfill $\lozenge$
\end{exam}

\subsubsection*{Strongly regular graphs}
We now study strongly regular graphs. 
A $k$-regular graph with $n$ vertices is a strongly regular graph with parameters $srg(n,k,e,d)$ if
every two adjacent vertices have $e$ common neighbours and every two non-adjacent vertices have $d$ common neighbours.
Unitary Cayley graphs which are strongly regular are classified in \cite[Corollary 16]{AA} while  
unitary Cayley sum graphs $G_R^+$ without loops which are strongly regular were recently classified in \cite[Theorem~4.5]{RAR}. 
In both cases, the proofs are algebraic. 
As an application, we now give a simple unified spectral proof of these facts.

\begin{prop}[\cite{AA, RAR}] \label{srg GR}
	Let $R$ be a finite commutative ring with identity. 
	\begin{enumerate}[$(a)$]
		\item $G_R$ is a strongly regular graph if and only if $R$ is local, $R=\ff_q \times \ff_{q}$ with $q\ge 3$ or else $R=\Z_2^n$ for $n\ge 2$, in which case we have  
		\begin{align*}
		& G_{\ff_q}=srg(q,q-1,q-2,0), & &G_L = srg(r,r-m, r-2m, r-m), \\ 
		& G_{\ff_q \times \ff_q}= srg(q^2, (q-1)^2, (q-2)^2, (q-1)(q-2)), & &G_{\Z_2^n}=srg(2n,1,0,0),
		\end{align*}
	\text{ with $L$ a local ring which is not a field}. \msk
		
		\item $G_R^+$ is a strongly regular graph if and only if $R$ is local with \emph{char}$(R/\frak m) = 2$ or $R=\ff_{2^n} \times \ff_{2^n}$ or else $R=\Z_2^n$ for $n\ge 2$.
	\end{enumerate}
\end{prop}

\begin{proof}
It is well-known that strongly regular graphs (srg) correspond to connected simple regular graphs having exactly 3 different eigenvalues with only one exception, the graph $mK_a$ (disjoint union of $m$ complete graphs $K_a$) with $m\ge 1$. 
This is the unique strongly regular graph with 2 eigenvalues, and it is disconnected when $m\ge2$. 

($a$) Suppose first that $(R,\frak m)$ is local.
By \eqref{GRlocal}, we have that $G_R=K_{r}$ or $G_R=K_{\frac{r}{m}\times m}$ with $m=|\frak m|$ depending whether $R$ is a field or not, respectively. 
By \eqref{spec GR local} and the previous comment, $G_R$ is a strongly regular graph. Moreover,
it is known that $K_r=srg(r,r-1,r-2,0)$ and $K_{m \times a}=srg(ma, (m-1)a, (m-2)a, (m-1)a)$ with $m,a>1$.
In particular, we have $K_{\frac{r}{m}\times m}=srg(r,r-m, r-2m, r-m)$.

Now, if $R=\ff_q\times \ff_q$, then $G_R$ is connected with spectrum 
\begin{equation} \label{spec Gfqfq}
Spec(G_{\ff_q \times \ff_q}) = \{ [(q-1)^2]^1, [1]^{(q-1)^2}, [-(q-1)]^{2(q-1)} \}
\end{equation}
by \eqref{spec GR local}. Hence, it is strongly regular with parameters 
$$G_{\ff_q \times \ff_q} = K_q \otimes K_q =  srg(q^2, (q-1)^2, (q-2)^2, (q-1)(q-2)).$$
On the other hand, if $R=\Z_2^n$, then $G_R=nK_2$. Since $mK_a = srg(ma, a-1,a-2,0)$ with $m,a>1$ we have that 
$G_{\Z_2^n}=srg(2n,1,0,0)$.

We now prove that these are the only possibilities. We thus assume that $G_R$ is a strongly regular graph with $R$ non-local.
If $\mu(\G)$ denotes the number of different eigenvalues of $\G$
then 
\begin{equation}\label{mues}
\mu(\G_1\otimes\G_2)\ge \max\{\mu(\G_1),\mu(\G_2)\}.
\end{equation}
The eigenvalues of $G_R$ are the product of the eigenvalues of $G_{R_1}$ and $G_{R_2}$.
By direct calculation, 
if $R_1,R_2$ are different local rings then $\mu(G_{R_1\times R_2})\ge 4$, by \eqref{spec GR local}.
So, if $G_R$ is a strongly regular graph and $R=R_1\times\cdots\times R_n$ with $n\ge 2$ 
then $R_i=R_j$ for all $i,j=1,\ldots,n$, since $\mu(G_R)\le 3$.
On the other hand, notice that if  $S_1,S_2$ are local rings such that 
some $S_i$ is not a field then $\mu(G_{S_1\times S_2})\ge 4$. 
Therefore, if $G_R$ is strongly regular then $R=\ff_{q}^n$ with $\ff_q$ a finite field.
By direct calculation, if $q>2$ then $\mu(G_{\ff_q^{3}})=4$, by \eqref{mues}, and 
$\mu(G_{\ff_q^{n}}) \ge 4$ if $n > 3$. Therefore, $n=2$ in this case.

$(b)$ If $G_R^+$ is strongly regular then it is simple, and this can only happen if $E(R) \ne \{0\}$ by ($b$) in 
Lemma \ref{lem nonsym}, and hence $|R|$ must be even. If $O(R) =\{0\}$ then $G_R^+=G_R$ by 
Lemma \ref{lem nonsym} and the result follows from ($a$). In general, if $O(R) \ne \{0\}$, by Lemma \ref{lem nonsym} and Theorem \ref{XRR* equinoiso} we have that 
$G_R^+$ has loops if $E(R)=0$ or else 
$\mu(G_R^+)> \mu(G_R)\ge 3$ if $E(R)\ne 0$, and hence we are in the cases in ($a$) with $r$ even.
\end{proof}

\section{Energy and complementary graphs} 
Here we address the computation of the energies of the graphs studied in the previous sections as well as of their complements. 
We recall that the energy of a graph $\G$ with eigenvalues $\{\lambda_1, \ldots, \lambda_n \}$ is given by 
	\begin{equation} \label{energy}
	E(\Gamma) = \sum_{1\le i \le n} |\lambda_i|.
	\end{equation} 
Furthermore, we are interested in determining which of these Cayley graphs are equienergetic with their own complements (in the non self-complementary case). 
We will find some of them and, as a consequence, we will exhibit triples $\{G_R, G_R^+, \bar G_R\} $ of equienergetic non-isospectral unitary Cayley graphs.

We recall that if $\G$ is a $k$-regular graph of $n$ vertices with non-principal eigenvalues $\{ \lambda \}$ the complementary graph $\bar \Gamma$ is an $(n-k-1)$-regular graph with non-principal eigenvalues $\{ -1-\lambda \}$. Also, 
the complement of a Cayley graph $\Gamma = X(G,S)$ is the Cayley graph $\bar \Gamma = X(G,S^c \smallsetminus\{0\})$.

The energy of a unitary Cayley graph $G_R=X(R,R^*)$ 
and that of its complement $\bar G_R$ are  already known. They were computed by Kiani et al in 2009 (\cite{Ki+}, see also \cite{Ak+}). 
If $R=R_1 \times \cdots \times R_s$ is as in \eqref{artin desc} then 
\begin{equation} \label{energy GRs}
\begin{split}
E(G_R) &= 2^s |R^*|, \\
E(\bar G_R) &= 2(|R|-1) + (2^s-2) |R^*| - \prod_{i=1}^s \tfrac{r_i}{m_i} + \prod_{i=1}^s (2- \tfrac{r_i}{m_i}),
\end{split}
\end{equation}
where $r_i=|R_i|$ and $m_i=|\frak m_i|$.
In this way, the condition $E(G_R) = E(\bar{G}_R)$ for equienergy between the graph and its complement is, by \eqref{energy GRs}, given by
\begin{equation} \label{equien G bar G}
2(|R|-|R^*|-1)= \prod_{i=1}^s \tfrac{r_i}{m_i} - \prod_{i=1}^s (2- \tfrac{r_i}{m_i}).
\end{equation}

We now classify all unitary Cayley graphs $G_R$ which are equienergetic with their complements,
for $R$ a finite commutative ring with only one or two factors in the artinian decomposition, i.e.\@ for $s=1,2$ in 
\eqref{artin desc}, or $R$ a product of at most 3 finite fields. 

\begin{thm} \label{s12}
Let $R$ be a finite commutative ring with unity. 
\begin{enumerate}[$(a)$]
	\item If $R$ is local with maximal ideal $\frak m$ then $E(G_R) = E(\bar G_R)$ if and only if 
		$|R|=m^2$ where $m=|\frak m|$ is a prime power (i.e.\@ $R$ is not a field). \sk
		
	\item If $R=R_1 \times R_2$, then $E(G_R) = E(\bar G_R)$ if and only if $R_1=\ff_q$ and $R_2=\ff_{q'}$ are both finite fields (not necessarily distinct). \sk 
	
	\item If $R=\ff_{q_1}\times \ff_{q_2} \times \ff_{q_3}$ is the product of 3 finite fields 
	with $q_1\le q_2\le q_3$, then $E(G_R) = E(\bar G_R)$ if and only if $(q_1,q_2,q_3)\in \{(3,5,5),(4,4,4)\}$.
\end{enumerate} 
Furthermore, all the equienergetic pairs $\{G_R, \bar G_R\}$ given by $(a)$, $(b)$ and $(c)$ are non-isospectral, with  the only exception of $R=\ff_3 \times \ff_3$.
\end{thm}

\begin{proof}
In all the cases, we first determine the equienergetic pairs and then discard the isospectral ones, if any.

\sk
($a$) Let $r=|R|$ and $m=|\frak m|$.
Notice that $|R^*|=r-m$ in this case. By \eqref{equien G bar G}, $G_R$ and $\bar G_R$ are equienergetic if and only if 
$$2(m-1) = \tfrac{r}{m}-(2-\tfrac{r}{m}) = 2(\tfrac rm -1),$$
which holds if and only if $r=m^2$. 
Since $R/\frak m$ is a field, $r$ and $m$ are prime powers. 

Also, $G_R$ has regularity degree $\kappa=|R^*|=r-m=m(m-1)$ while $\bar G_R$ has regularity degree $r-\kappa-1= m-1$, and hence $G_R$ and $\bar G_R$ are not isospectral to each other.

\sk
($b$) Now, assume that $R=R_1\times R_2$. 
Let $\frak m_1, \frak m_2$ be the maximal ideals of $R_1$ and $R_2$, respectively, and 
put $r_i=|R_i|$, $m_i=|\frak m_i|$ and $q_i=\frac{r_i}{m_i}$ for $i=1,2$. Thus, \eqref{equien G bar G} takes the form 
$$2(r_1r_2 - (r_1-m_1)(r_2-m_2)-1) = \tfrac{r_1 r_2}{m_1 m_2}-(2-\tfrac{r_1}{m_1})(2-\tfrac{r_2}{m_2}).$$
Thus, taking into account that $r_i=q_i m_i$ we have
$$2(m_1m_2q_1 q_2- m_1m_2(q_1-1)(q_2-1))-2= q_1 q_2- (2-q_1)(2-q_2).$$
Notice that if we write $2-q_i=1-(q_i-1)$ for $i=1,2$, then we have that  
$$(2-q_1)(2-q_2)=1-(q_1-1)-(q_2-1)+(q_1-1)(q_2-1)$$ 
and hence $(2m_1m_2-1)(q_1q_2-(q_1-1)(q_2-1))=1+(q_1-1)+(q_2-1)$.
In this way, we get 
$$(2m_1m_2-1)(q_1+q_2-1)=q_1+q_2-1.$$
Since $q_i\ge 1$ for $i=1,2$, by cancellation we obtain $2m_1m_2-1=1$
which clearly holds if and only if $m_1=m_2=1$. Therefore $R_1$ and $R_2$ are finite fields, as asserted.

Regarding the isospectrality, we first discard 	$R=\ff_3 \times \ff_3$. In fact,  
since $G_{\ff_3\times \ff_3}=K_{3}\otimes K_3$ (or by \eqref{spec Gfqfq}) we have 
$$\mathrm{Spec}(G_{\ff_3\times \ff_3})=\{[4]^1,[2]^4,[-1]^4\}=\mathrm{Spec}(\bar{G}_{\ff_3\times \ff_3}).$$
Now, suppose that $G_{\ff_{q_1}\times \ff_{q_2}}$ and $\bar{G}_{\ff_{q_1}\times \ff_{q_2}}$ are isospectral with $(q_1,q_2)\ne(3,3)$.
In particular, they have the same regularity degree, and hence we have $q_1q_2-1 = 2(q_1-1)(q_2-1)$. From this, we get
\begin{equation}\label{q1q2}
1 + \tfrac{3}{q_1 q_2} = 2 (\tfrac{1}{q_1} + \tfrac{1}{q_2} ).
\end{equation} 
Observe that the left side of the above equality is greater than $1$, so necessarily $\frac{1}{q_1} +\frac{1}{q_2}>\frac 12$. 
By taking into account that $q_1\le q_2$, we obtain that $q_1\le 3$. 
If $q_1=2$, then the equality \eqref{q1q2} takes the form $1+\frac{3}{2q_2}=1+\frac{2}{q_2}$ which 
has no solutions. 
If $q_1=3$, by \eqref{q1q2} we have that $1+\frac{1}{q_2}=\frac{2}{3}+\frac{2}{q_2}$ and thus $q_2=3$, but this is a contradiction since $(q_1,q_2)\ne(3,3)$.
Therefore $G_{\ff_{q_1}\times \ff_{q_2}}$ and $\bar{G}_{\ff_{q_1}\times \ff_{q_2}}$ are non-isospectral for $(q_1,q_2) \ne (3,3)$, as desired.

\sk
($c$) Now, if $R=\ff_{q_1}\times \ff_{q_2} \times \ff_{q_3}$, by proceeding similarly as before we have that
$E(G_R)=E(\bar{G}_{R})$ if and only
$$(q_1-1)(q_2-1)+(q_1-1)(q_3-1)+(q_2-1)(q_3-1) = (q_1-1)(q_2-1)(q_3-1).$$
So, dividing the above equality by $(q_1-1)(q_2-1)(q_3-1)$ we see that it is enough to find the positive integer solutions (in prime powers) of
\begin{equation} \label{eq equien s=3}
\tfrac{1}{x_1-1} + \tfrac{1}{x_2-1} + \tfrac{1}{x_3-1}=1.
\end{equation}

Without loss of generality, we can assume that $2\le x_1\le x_2\le x_3$. Since $\frac{1}{x_i-1}>0$ for all $i$, we have 
$x_1\ge 3$.
On the other hand if $x_1\ge 5$ then 
$\frac{1}{x_1-1} + \frac{1}{x_2-1} + \frac{1}{x_3-1} \le \frac{3}{x_1-1} \le \frac{3}{4}$, 
which cannot occur.
So we obtain that $x_1<5$ and, hence, $x_1 \in \{3,4\}$.
If $x_1=3$, then \eqref{eq equien s=3} takes the form
$$\tfrac{1}{x_2-1} + \tfrac{1}{x_3-1} = \tfrac{1}{2}.$$
Notice that necessarily we have that $x_2\le 5$, since if $x_2>5$ then
$\frac{1}{x_2-1} + \frac{1}{x_3-1} \le \frac{2}{x_2-1} < \frac{2}{4} = \frac{1}{2}$,
which is impossible.
Furthermore, $x_2 \ne 3,4$. In fact, if $x_2=3$ we arrive at $\frac{1}{x_3-1}=0$ and if $x_2=4$ we obtain that $x_3$ is not an integer, and both cases are not allowed.
Thus, the only possibility is $x_2=x_3=5$.  
A similar argument shows that if $x_1=4$, then $x_2=x_3=4$ is the only solution of \eqref{eq equien s=3}. 
This shows that  if $R$ is a product of three finite fields, then $R=\ff_3 \times \ff_5 \times \ff_5$ or 
$R=\ff_4 \times \ff_4 \times \ff_4$.

Finally, $G_{\ff_3 \times \ff_5 \times \ff_5}$ (resp.\@ $G_{\ff_4 \times \ff_4 \times \ff_4}$) is not isospectral to $\bar G_{\ff_3 \times \ff_5 \times \ff_5}$ (resp.\@ $\bar G_{\ff_4 \times \ff_4 \times \ff_4}$) since they have different degrees of regularity.
\end{proof}

We illustrate the proposition with some local rings. 

\begin{exam} \label{ex m-partite}
($i$) If $R$ is a quadratic extension of $\Z_p$, with $p$ prime, then $E(G_R)=E(\bar G_R)$.
In fact, we have $r=p^{2t}$, $m=p^t$. In particular, since the maximal ideal of $\Z_{p^{2t}}$ is $\Z_{p^{2t-1}}$, the condition $r=m^2$ in ($a$) of the proposition holds if and only if $2t-1=t$, i.e.\@ $t=1$.
Thus, we can take 
\begin{equation} \label{Zp2}
R=\Z_{p^2} \quad \text{ or } \quad R=\Z_p[x]/(x^2).
\end{equation}
The Cayley graphs $G_R$ for these rings are isomorphic by \eqref{GRlocal}. 
By ($a$) of Theorem~\ref{s12} we have that $G_{R}$ and $\bar G_{R}$ are equienergetic and non-isospectral.

\noindent ($ii$) More generally, $R= \ff_{p^t}[x]/(x^2)$ 
is a local ring for any prime $p$ and $t \in \N$, whose maximal ideal is 
$(x) = x\ff_{p^t}$. Thus, we have $r=m^2$ and hence by ($a$) in the proposition, $G_R$ and $\bar G_R$ are equienergetic non-isospectral graphs. 
\hfill $\lozenge$
\end{exam}

\begin{exam} \label{ex m=4}
Now we consider local rings having the smallest possible maximal ideals. If $m=2$ or $m=3$ the only rings are those given in \eqref{Zp2} with $p=2$ or $3$ respectively. These graphs are trivial, since $G_{\Z_4}=K_4$ and $G_{\Z_9} = \bar C_9$.
There are seven local rings with $m=4$ (\cite{AA}); namely 
\begin{gather*}
R_1=\Z_2[x]/(x,y)^2, \quad R_2=\Z_2[x]/(x^3), \quad R_3=\Z_4[x]/(2x,x^2), \quad R_4= \Z_4[x]/(2x,x^2-2), \\ 
R_5=\Z_8, \quad R_6= \ff_4[x]/(x^2) \quad \text{and} \quad R_7=\Z_4[x]/(x^2+x+1).  
\end{gather*}
It is clear that $|R_2|=|R_3|=|R_4|=|R_5|=8$, and hence $r\ne m^2$, so $G_{R_i}$ is not equienergetic to $\bar G_{R_i}$ for $i=2,3,4,5$ by Theorem \ref{s12}. On the other hand, $|R_1|=|R_6|=|R_7|=16$ and hence, by Theorem \ref{s12}, $G_{R_i}$ and $\bar G_{R_i}$ are equienergetic non-isospectral graphs for $i=1,6,7$. 
All the graphs $G_{R_i}$, $i=1,6,7$, are isomorphic to $K_{4 \times 4} = K_{4,4,4,4}$, by \eqref{GRlocal}.
\hfill $\lozenge$
\end{exam}

By taking into account particular instances of Theorem \ref{s12} 
we obtain graphs equienergetic with their complements from known families of graphs; namely, from complete $m$-multipartite graphs $K_{m \times m}$ of $m^2$ vertices and crown graphs $H_{m,m}$ of $2m$ vertices. 
We recall that the crown graph $H_{m,m}$ can be seen as $K_{m,m}$ with a perfect matching removed (or also as $K_m \otimes K_2$, the complete bipartite double of $K_m$, as $\overline{K_m \times K_2}$, or as the Kneser graph $K_{n,1}$).
 
\begin{coro} \label{teo crowns}
Let $m=p^t$ be a prime power. 
\begin{enumerate}[$(a)$]
	\item $\{K_{m\times m}, mK_m\}$ are equienergetic non-isospectral graphs of $m^2$ vertices for any $m\ge 2$. \msk 
	
	\item $\{H_{m,m}, mK_2\}$ are equienergetic non-isospectral graphs of $2m$ vertices for any $m\ge 3$.
\end{enumerate}
In each pair, one of the graph is connected while the other not.
\end{coro}

\begin{proof}
($a$) If $(R,\frak m)$ is a local ring, then $G_R = K_{\frac rm \times m}$ is the complete $\frac rm$-multipartite graph by \eqref{GRlocal}, where $m=|\frak m|$. By Theorem \ref{s12}, $E(G_R)=E(\bar G_R)$ if and only if $r=m^2$ with $m>1$.
Hence, in this case, $G_R = K_{m\times m}=K_{m,\ldots,m}$ and  
$\{ K_{m\times m}, \bar K_{m\times m}\}$ is a pair of equienergetic non-isospectral graphs, and it is clear that $\bar K_{m\times m} = mK_n$.
	
\noindent
($b$) If $R=\ff_q \times \ff_{q'}$, then $G_R$ is both a $q$-partite and $q'$-partite graph. In the particular case when $R=\ff_2 \times \ff_q$, the graph $G_R$ is isomorphic to the complete bipartite graph $K_{q,q}$ with a perfect matching removed, i.e.\@ $G_R=H_{q,q}$. 
By Theorem~\ref{s12} again, $\{ H_{m,m}, \bar H_{m,m}\}$ is a pair of equienergetic non-isospectral graphs and obviously $\bar H_{m,m} = mK_2$. Since $H_{2,2}=\bar H_{2,2}=2K_2$, we must take $m\ge 3$. 	
\end{proof}

Note that for $m=2$ in $(a)$ of the corollary, we recover the minimum example given in the Introduction, since $K_{2,2}=C_4$ and $2K_2 = K_2 \otimes K_2$.

 \msk

Now, we produce triples $\{ G_R, G_R^+ ,\bar G_R\}$ of integral equienergetic non-isospectral graphs, where $R$ has $s\le 3$ factors in the artinian decomposition. For $s=1,2$ we get infinite families of such triples.
In the previous notations we have.

\begin{coro} \label{ternas GG+Gbar}
Let $R$ be one of the following rings:
\begin{enumerate}[$(a)$]
	\item $R$ is a local ring with $r=m^2$ odd, \msk 
	
	\item $R=\ff_{q_1} \times \ff_{q_2} \ne \ff_3 \times \ff_3$ with $q_1, q_2 \ge 3$ at least one odd, or \msk 
	
	\item $R=\ff_3 \times \ff_5 \times \ff_5$. 
\end{enumerate}
Then, $\{ G_R, G_R^+ ,\bar G_R\}$ are integral equienergetic non-isospectral graphs.  
\end{coro}

\begin{proof}
The result is straightforward from Proposition \ref{Spec GR+} and Theorems \ref{XRR* equinoiso} and \ref{s12}, since $R=\ff_4^3$ is not of odd type.
\end{proof}

\begin{exam}
Infinite triples $\{G_R, G_R^+ ,\bar G_R \}$ of integral equienergetic non-isospectral graphs are given by taking $R=\Z_{(2k+1)^2}$ for any $k\in \N$, $R=\ff_3 \times \ff_q$ or $R=\ff_q \times \ff_q$ for any prime power $q  \ge 4$. 
The smallest such triples are $\{ G_{\Z_9}, G_{\Z_9}^+, \bar G_{\Z_9} \}$ for $R$ local or $\{G_{\ff_3 \times \ff_4}, G_{\ff_3 \times \ff_4}^+, \bar G_{\ff_3 \times \ff_4}\}$ for $R$ non-local.
Note that $G_{\Z_9}^+$ is loopless while $G_{\ff_3 \times \ff_4}^+$ has loops.

We now study these minimal examples in more detail by computing their spectra and their energies explicitly. 
By \eqref{spec GR local} and Proposition~\ref{Spec GR+} 
we have that 
$Spec(G_{\Z_9}) = \{ [6]^1, [0]^6, [-3]^2 \}$, $Spec(G_{\Z_9}^+) = \{ [6]^1, [3]^1, [0]^6, [-3]^1 \}$ 
and $Spec(\bar G_{\Z_9}) = \{ [2]^3, [-1]^6 \}$. All the graphs have energy $E=12$. 
Note that $G_{\Z_9}=K_{3,3}$, $\bar G_{\Z_9}=3K_3$ and that $G_{\Z_9}^+$ is connected non-bipartite with loops.

To compute the spectra for the graphs with $R=\ff_3 \times \ff_4$, note that $G_R= G_{\ff_3} \otimes G_{\ff_4}$ and 
$G_R^+= G_{\ff_3}^+ \otimes G_{\ff_4}^+ = G_{\ff_3}^+ \otimes G_{\ff_4}$, and that the eigenvalues of the Kronecker product are the product of the eigenvalues of the factors. 
Hence, since we have $Spec(G_{\ff_3})= \{ [2]^1, [-1]^2 \}$, $Spec(G_{\ff_3}^+)= \{ [2]^1, [1]^1, [-1]^1 \}$ and $Spec(G_{\ff_4})= \{ [3]^1, [-1]^3 \}$, we deduce that 
\begin{align*}
& Spec(G_{\ff_3 \times \ff_4}) = \{ [6]^1, [1]^6, [-2]^3, [-3]^2 \}, \\ 
& Spec(G_{\ff_3 \times \ff_4}^+) = \{ [6]^1, [3]^1, [1]^3, [-1]^3, [-2]^3, [-3]^1 \}, \\ 
& Spec(\bar G_{\ff_3 \times \ff_4}) = \{ [5]^1, [2]^2, [1]^3, [-2]^6 \}. 
\end{align*}
All the graphs are connected, non-bipartite and have energy $E= 24$.
Further, note that $G_{\ff_3 \times \ff_4}^+$ is loopless. 
\hfill $\lozenge$
\end{exam}

\section{Equienergetic non-isospectral Ramanujan graphs}
In this section we consider Ramanujan graphs. We will show that there exist infinite families of equienergetic non-isospectral pairs of graphs, both Ramanujan or one being Ramanujan and the other not. 
Furthermore, we will characterize all pairs and triples of equienergetic non-isospectral Ramanujan pairs of the form $\{G_R,G_R^+\}$, $\{G_R, \bar G_R\}$ and $\{G_R,G_R^+, \bar G_R\}$.
We will use the results of the previous sections and the known characterization of Ramanujan unitary Cayley graphs $G_R$
due to Liu-Zhou (\cite{LZ}, 2012). For convenience, we will distinguish the cases when $R$ is a local ring or not.

Recall that a connected $n$-regular graph $\Gamma$ is Ramanujan if 
\begin{equation} \label{Ram eq}
\lambda(\Gamma) \le 2\sqrt{n-1}, 
\end{equation}
where $\lambda(\Gamma) = \max \{ |\lambda| : \lambda_0  \ne \lambda \in Spec(\G) \}$ is the greatest absolute value of the non-principal eigenvalues.   
There is a more general definition which applies to regular digraphs with loops.
A $k$-regular directed graph $\G$ is Ramanujan if it satisfies \eqref{Ram eq} and also its adjacency matrix can be diagonalized by a unitary matrix. However, the adjacency matrix of a sum graph $X^+(G,S)$ with $G$ abelian is diagonalizable by a unitary matrix (see Proposition 2 in \cite{L}). So, in the case we are interested in, namely $G_R^+$, this condition is automatic and we just have to check \eqref{Ram eq}. 
Since $G_R^+$ is a strongly almost symmetric graph by Theorem \ref{XRR* equinoiso}, the graphs $G_R$ and $G_R^+$ are both Ramanujan or both not Ramanujan, respectively, 
since 
\begin{equation} \label{ram GG+}
\lambda(G_R)=\lambda(G_R^+).
\end{equation}

\subsubsection*{$R$ a local ring}
If $R$ is a local ring with maximal ideal $\frak m$, by Theorems 11 and 15 in \cite{LZ}, we respectively have that $G_R$ is Ramanujan if and only if 
\begin{equation} \label{cond ramanujs}
(i) \quad r=2m \qquad \text{or else} \qquad (ii) \quad r\ge (\tfrac m2 +1)^2 \quad \text{and} \quad m\ne 2,
\end{equation}
where $r=|R|$, $m=|\frak m|$, and that $\bar G_R$ is always Ramanujan.

Combining this with previous results we obtain the following. 
\begin{thm} \label{teo ram gr}
Let $(R, \frak m)$ be a finite commutative local ring with $|R|=r$ and $|\frak m|=m$. Then, 
\begin{enumerate}[$(a)$]
	\item $\{ G_R, G_R^+ \}$ are equienergetic non-isospectral graphs if $r$ is odd,
	and both Ramanujan if and only if $r$ and $m$ satisfy $(ii)$ in \eqref{cond ramanujs}. \msk
	
	\item $\{ G_R, \bar G_R \}$ are equienergetic non-isospectral graphs 
	if and only if $r = m^2$, where $\bar G_R$ is Ramanujan. The graph $G_R$ is Ramanujan provided that $r$ and $m$ satisfy 	\eqref{cond ramanujs}. 
\end{enumerate}
\end{thm}

\begin{proof}
($a$) If $r$ is odd, $G_R$ and $G_R^+$ are equienergetic non-isospectral graphs by Proposition~\ref{Spec GR+}. By \eqref{ram GG+}, $G_R$ and $G_R^+$ are both Ramanujan or both not Ramanujan. Moreover, $G_R$ is Ramanujan if ($ii$) in \eqref{cond ramanujs} holds. 

\noindent
($b$) We know that the graphs $G_R$ and $\bar G_R$ are non-isospectral and, by ($a$) in Theorem \ref{s12}, they are equienergetic  
if and only if $r = m^2$. The remaining assertions follow by \eqref{cond ramanujs}. 
\end{proof}

The following is a direct consequence of the previous proposition.
\begin{coro} \label{coro gr ram}
Let $(R,\frak m)$ be a local ring with $r=|R|$ and $m=|\frak m|$. Then, $\{ G_R, G_R^+, \bar G_R \}$ are equienergetic non-isospectral Ramanujan graphs if and only if $r=m^2$ is odd. 
\end{coro}

\begin{proof}
If $\{ G_R, G_R^+, \bar G_R \}$ are equienergetic non-isospectral Ramanujan graphs then $r=m^2$ with $m>1$ and $r$ odd, by 
Theorem \ref{teo ram gr}. For the converse, if $r=m^2$ is odd we have that $R$ is non-trivial and, by the same theorem, 
$\{G_R, G_R^+, \bar G_R\}$ are mutually equienergetic and each of the pairs $\{G_R, G_R^+\}$ and $\{G_R, \bar G_R\}$ are formed by non-isospectral graphs. That $\{G_R^+, \bar G_R\}$ are non-isospectral follows from the fact that they have different degrees of regularity (since $m>1$).
Also, $\bar G_R$ is Ramanujan and the graphs $G_R$ and $G_R^+$ are Ramanujan if and only if $r$ and $m$ satisfy the second condition in \eqref{cond ramanujs}, which is clearly implied by $r=m^2$. 
\end{proof}

\begin{exam}
Let $R=\Z_{p^2}$ or $R=\ff_p[x]/(x^2)$ with $p$ an odd prime. Then, by Example~\ref{ex m-partite} ($ii$) and the previous corollary, $\{ G_R, G_R^+, \bar G_R \}$ is a triple of equienergetic non-isospectral Ramanujan graphs. \hfill $\lozenge$
\end{exam}

\subsubsection*{$R$ a non-local ring}
Consider $R$ a finite commutative ring with $1\ne 0$ which is not local. 
We now give the following characterizations for the pairs $\{ G_R, G_R^+ \}$ and $\{ G_R, \bar G_R \}$ to be equienergetic non-isospectral Ramanujan graphs.

\begin{thm} \label{teo ram gr nonlocal}
	Let $R$ be a commutative finite ring which is non-local. 
	\begin{enumerate}[$(a)$]
		\item If $R$ is of odd type then $\{ G_R, G_R^+ \}$ are equienergetic non-isospectral Ramanujan graphs 
		if and only if $R$ is one of the following 7 types of rings:
$$\ff_3 \times \ff_3, \:\: \ff_3 \times \ff_4, \:\: \ff_3 \times \ff_3 \times \ff_3, \:\: 
\ff_3 \times \ff_3 \times \ff_4, \:\: 
\ff_3 \times \Z_9, \:\: \ff_3 \times \Z_3[x] / (x^3), \:\: 
\ff_{q_1} \times \ff_{q_2}$$ 
where $q_1, q_2$ are at least one odd and satisfy
	\begin{equation} \label{q1q2conds}
		3 \le q_1 \le q_2 \le 2(q_1+\sqrt{(q_1-2)q_1})-1.
	\end{equation} 
 In particular, there are infinite pairs $\{ G_R, G_R^+ \}$ as above. \msk

		\item If $R=R_1 \times R_2$ with $R_1, R_2$ local rings then $\{ G_R, \bar G_R \}$ are equienergetic non-isospectral Ramanujan graphs if and only $R$ is one of the following 17 product of finite fields:
\begin{gather*}
\ff_3 \times \ff_q \text{ with } q=4,5,7, \qquad \ff_4 \times \ff_q \text{ with } q=4,5,7,8, \qquad 
\ff_5 \times \ff_q \text{ with } q=5,7,8, \\ 
\ff_7 \times \ff_q \text{ with } q=7,8,9, \qquad \ff_8 \times \ff_8, \qquad 
\ff_8 \times \ff_9, \qquad \ff_9 \times \ff_9, \qquad \ff_{11} \times \ff_{11}.
\end{gather*}		
\end{enumerate}
\end{thm}

\begin{proof}
	($a$) Since $R$ is of odd type, the graphs $G_R$ and $G_R^+$ are equienergetic and non-isospectral by Theorem 
	\ref{XRR* equinoiso}. Also, the Ramanujan graphs $G_R$ (and hence $G_R^+$) are obtained from the characterization given in Theorem 12 in \cite{LZ}, in the case $R$ is of odd type. Thus, the only possibilities are, in this theorem, to take:
	\begin{itemize}
		\item ($b$) with $s=2, 3$, thus giving $R=\ff_3 \times \ff_3, \, \ff_3 \times \ff_3 \times \ff_3$, or \msk 
		
		\item ($c$) with $s=2, 3$, hence giving $R=\ff_3 \times \ff_4, \, \ff_3 \times \ff_3 \times \ff_4$, or \msk 

		\item ($e$) with $s=2$, thus giving $R= \ff_3 \times \Z_9, \, \ff_3 \times \Z_3[x]/(x^3)$, or else \msk
		
		\item ($f$) with $s=2$, hence giving $R=\ff_{q_1} \times \ff_{q_2}$ with $q_1, q_2$ satisfying \eqref{q1q2conds}.	\end{itemize}
(The case ($g$) with $s=2$ is contained in the previous one.)
Thus, $R$ must be one of the seven cases in the statement. 
The remaining assertion follows from the fact that \eqref{q1q2conds} is equivalent to $\frac{(q_2+1)^2}{q_2-1} \le 4q_1$, which clearly has infinitely many solutions.

\msk	
	
	\noindent
	($b$) By Theorem \ref{s12}, the graphs $G_R, \bar G_R$ are equienergetic and non-isospectral
		if and only if $R=\ff_{q_1} \times \ff_{q_2}$ is the product of two finite fields with $(q_1,q_2) \ne (3,3)$. 
	In this situation, $G_R$ is Ramanujan if we are in the conditions ($c$), ($d$) or ($f$) in Theorem 12 of \cite{LZ} (with $s=2$). Hence 
\begin{equation} \label{GRGR+rams nolocal}
R=\ff_3 \times \ff_4, \qquad R=\ff_4 \times \ff_4 \qquad \text{or} \qquad  
R=\ff_{q_1} \times \ff_{q_2}
\end{equation} 
with $q_1, q_2$ prime powers satisfying \eqref{q1q2conds}. 
	
	On the other hand, to find which $\bar G_R$ is Ramanujan with $R$ a product of two finite fields we use the characterization given in Theorem 16 in \cite{LZ}. Item ($a$) in this theorem gives $R=\ff_2 \times \ff_2$ while item ($c$) gives $R=\ff_2 \times \ff_q$ with $q$ satisfying $q^2-8q+8\le 0$, which can only hold for $q=3,4,5$. 
	Item ($b$) is empty for $s=2$. Finally, item 
	($d$) gives $R=\ff_{q_1} \times \ff_{q_2}$ with $3\le q_1 \le q_2$ and 
	\begin{equation} \label{conditions rams2}
			2(q_1-2)+q_2 \le \sqrt{(2q_1-3)^2 + (4q_1q_2-9)},
	\end{equation}
	which is equivalent to $q_2(q_2-8) \le 4(q_1-4)$. 
	This inequality has a finite number of solutions $(q_1,q_2)$ given by $(3,q_2)$ with $q_2=3,4,5,7$, 
	$(4,q_2)$ with $q_2=4,5,7,8$, $(5,q_2)$ with $q_2=5,7,8$, $(7,q_2)$ with $q_2=7,8,9$, $(8,8)$, $(8,9)$, $(9,9)$ or $(11,11)$.
Putting all these things together, we have that $\bar G_R$ is Ramanujan if and only if $R$ is one of the following 21 rings:
\begin{equation} \label{GRGRbar rams nolocal}
\begin{aligned}
&\ff_2 \times \ff_2, \quad \ff_2 \times \ff_3, \quad  \ff_2 \times \ff_4, \quad \ff_2 \times \ff_5, 
\quad \ff_3 \times \ff_4, \quad \ff_3 \times \ff_5, \quad  \ff_3 \times \ff_7,\\  
&\ff_4 \times \ff_4, \quad \ff_4 \times \ff_5, \quad  \ff_4 \times \ff_7, \quad \ff_4 \times \ff_8, 
\quad \ff_5 \times \ff_5, \quad \ff_5 \times \ff_7, \quad  \ff_5 \times \ff_8,\\  
&\ff_7 \times \ff_7, \quad \ff_7 \times \ff_8, \quad  \ff_7 \times \ff_9, \quad \ff_8 \times \ff_8, 
\quad \ff_8 \times \ff_9, \quad \ff_9 \times \ff_9, \quad  \ff_{11} \times \ff_{11}. 
\end{aligned}
\end{equation}  
All the pairs in \eqref{GRGRbar rams nolocal} with $q_1 \ge 3$ satisfy \eqref{conditions rams2}. Therefore, the pairs of equienergetic non-isospectral Ramanujan graphs of the form $\{G_R, \bar G_R\}$ are those satisfying both  
\eqref{GRGR+rams nolocal} and \eqref{GRGRbar rams nolocal}, and hence are as stated.
\end{proof}

\begin{exam} \label{ej 7.5}
Suppose $R_1=\ff_3 \times \ff_q$ and $R_2=\ff_4 \times \ff_q$ and $R_3=\ff_5 \times \ff_q$. By \eqref{q1q2conds} in the theorem, $\{G_{R_i}, G_{R_i}^+\}$ is a pair of equienergetic non-isospectral Ramanujan graphs if and only if $q=3,4,5,7,8$ for $i=1$, $q=5,7,9,11$ for $i=2$ and $q=5,7,8,9,11,13$ for $i=3$. \hfill $\lozenge$
\end{exam}

\begin{exam}
Now we want examples of $\{G_R,\bar G_R\}$ equienergetic non-isospectral, with one of the graphs Ramanujan and the other not.
By the previous example, we can take $R$ any of $\ff_3 \times \ff_8$, $\ff_4 \times \ff_9$, $\ff_4 \times \ff_{11}$, 
$\ff_5 \times \ff_9$, $\ff_5 \times \ff_{11}$ or $\ff_5 \times \ff_{13}$, since $G_R$ Ramanujan but $\bar G_R$ is not, for $R$ is not in the list \eqref{GRGRbar rams nolocal}. There are infinite examples of this kind, just take any ring $R=\ff_{q_1} \times \ff_{q_2}$ satisfying \eqref{conditions rams2} not in the list \eqref{GRGRbar rams nolocal}. On the other hand, a pair $\{G_R,\bar G_R\}$ of equienergetic non-isospectral graphs with $\bar G_R$ Ramanujan and $G_R$ not Ramanujan cannot exist.
In fact, if $\bar G_R$ is Ramanujan, it is one of the list \eqref{GRGRbar rams nolocal} and hence satisfies \eqref{GRGR+rams nolocal}, thus implying that $G_R$ is also Ramanujan.  
\hfill $\lozenge$
\end{exam}

As a straightforward consequence of Theorem \ref{teo ram gr nonlocal}, we can characterize all triples of equienergetic non-isospectral Ramanujan graphs of the form $\{ G_R, G_R^+, \bar G_R \}$ where $R$ is a non-local ring. In particular, $R$ can only be the product of 2 finite fields. 

\begin{prop} \label{prop tripla}
Let $R$ be a finite commutative non-local ring. Then, $\{ G_R, G_R^+, \bar G_R \}$ are equienergetic non-isospectral Ramanujan graphs if and only if $R$ is one of the following 14 rings: 
\begin{align*}
&\ff_3 \times \ff_4, \quad 
\ff_3 \times \ff_5, \quad 
\ff_3 \times \ff_7, \quad 
\ff_4 \times \ff_5, \quad 
\ff_4 \times \ff_7, \quad 
\ff_5 \times \ff_5, \quad 
\ff_5 \times \ff_7, \\
&\ff_5 \times \ff_8, \quad 
\ff_7 \times \ff_7, \quad 
\ff_7 \times \ff_8, \quad 
\ff_7 \times \ff_9, \quad 
\ff_8 \times \ff_9, \quad 
\ff_9 \times \ff_9, \quad 
\ff_{11} \times \ff_{11}.
\end{align*}
Furthermore, $8 \mid E(G_R)$ if $|R|$ is even and $16 \mid E(G_R)$ if $|R|$ is odd.
\end{prop}

\begin{proof}
We know by Theorem \ref{teo ram gr nonlocal} that if $R$ is one of the 14 rings in the statement, then the graphs $\{ G_R, G_R^+, \bar G_R \}$ are equienergetic non-isospectral Ramanujan graphs. We now prove the converse.

By ($a$) in Theorem \ref{teo ram gr nonlocal}, if $\{G_R, G_R^+\}$ are equienergetic non-isospectral and Ramanujan graphs, then $R$ has only 2 or 3 local factors, and in the case of 3 factors, they are all finite fields. If $R$ has two local factors, the rings satisfying both $(a)$ and $(b)$ in Theorem \ref{teo ram gr nonlocal} are just the rings in ($b$) of this theorem excluding $\ff_4 \times \ff_4$, $\ff_4 \times \ff_8$ and $\ff_8 \times \ff_8$ which are not of odd type. This gives all the rings in the statement. 
 
Let us see that $R$ cannot have 3 local factors. If this was the case  
then, by ($a$) of Theorem~\ref{teo ram gr nonlocal}, $\{G_R, G_R^+\}$ are equienergetic non-isospectral Ramanujan graphs if and only $R$ is $\ff_3 \times \ff_3 \times \ff_3$ or $\ff_3 \times \ff_3 \times \ff_4$. On the other hand, by 
Theorem \ref{s12}, if $\{G_R, \bar G_R \}$ are equienergetic and non-isospectral 
then $R$ is $\ff_3 \times \ff_5 \times \ff_5$ or $\ff_4 \times \ff_4 \times \ff_4$. 
Hence the case $R=R_1 \times R_2 \times R_3$ cannot happen.

The remaining assertion follows since, by \eqref{energy GRs}, 
we have $E(G_R) = 2^2|R^*|=4(q_1-1)(q_2-1)$ 
for $R=\ff_{q_1} \times \ff_{q_2}$. 
\end{proof}

In the particular case of $R=\Z_n$, the ring of integers modulo $n$, we have the following. 
\begin{coro} \label{coro ternas}
Let $R=\Z_n$ with $n$ odd.
Then, $\{ G_{R}, G_{R}^+, \bar G_{R} \}$ are equienergetic non-isospectral Ramanujan graphs if and only if
\begin{enumerate}[$(a)$]
	\item $R=\Z_{p^2}$ with $p$ an odd prime in the local case or, \sk
	
	\item $R \in \{ \Z_3 \times \Z_5, \: \Z_3 \times \Z_7, \: \Z_5 \times \Z_5, \: \Z_5 \times \Z_7, \: \Z_7 \times \Z_7,
	\: \Z_{11} \times \Z_{11}\}$ in the non-local case. 	
\end{enumerate}
\end{coro}

\begin{proof}
	It follows directly by the previous results in Sections 4 and 5. 
\end{proof}

Now, we list in Table 1 the smallest graphs giving a triple $\{G_R, G_R^+, \bar G_R \}$ of equienergetic Ramanujan graphs. 

We give the number of vertices $v$, the degrees of regularity $\kappa$ and $\bar \kappa$ of $\G$ and $\bar \G$ 
and the energy (see \eqref{energy GRs}). We list the graphs in ascending order of vertices (local rings first if tie).
Some triples $\{G_R, G_R^+, \bar G_R \}$ may contain an isospectral pair; we indicate this with an asterisk in the last column. 
Note that there are 3 pairs having the same energy, but the graphs in each pair are not equienergetic since they have different orders.
\renewcommand{\arraystretch}{.9}
\begin{table}[H] 
	\caption{Smallest graphs $\G$ such that $\{G_R, G_R^+, \bar G_R\}$ are equienergetic and Ramanujan} \vspace{-5mm}
	$$	\begin{tabular}{|c|c|c|c|c|c|}
	\hline
	graph & $v$ & $\kappa$ & $\bar \kappa$ & energy & iso\\
	\hline
	$G_{\Z_9}$ 						& 9  & 6 & 2  & 12 &  \\ \hline
	$G_{\ff_3 \times \ff_3}$ 		& 9  & 4 & 4  & 16 & $*$ \\ \hline	
	$G_{\ff_3 \times \ff_4}$ 		& 12 & 6 & 5  & 24 & \\ \hline
	$G_{\ff_3 \times \ff_5}$ 		& 15 & 8 & 6  & 32 & \\ \hline
	$G_{\ff_4 \times \ff_4}$ 		& 16 & 9 & 6  & 36 & $*$ \\ \hline
	$G_{\ff_4 \times \ff_5}$ 		& 20 & 12& 7  & 48 & \\ \hline
	$G_{\ff_3 \times \ff_7}$ 		& 21 & 12& 8  & 48 & \\ \hline
	$G_{\Z_{25}}$ 				 	& 25 & 20& 4  & 40 & \\ \hline
	$G_{\ff_5 \times \ff_5}$ 		& 25 & 16& 8  & 64 & \\ \hline
	$G_{\ff_4 \times \ff_7}$ 		& 28 & 18& 9  & 72 & \\ \hline
	$G_{\ff_4 \times \ff_8}$ 		& 32 & 21& 10 & 84 & $*$ \\ \hline
	$G_{\ff_5 \times \ff_7}$ 		& 35 & 24& 10 & 96 & \\ \hline
	\end{tabular} \qquad 
	\begin{tabular}{|c|c|c|c|c|c|}
	\hline
	graph & $v$ & $\kappa$ & $\bar \kappa$ & energy & iso\\
	\hline
	$G_{\ff_5 \times \ff_8}$ 		& 40 & 28& 11 & 112  & \\ \hline
	$G_{\Z_{49}}$ 					& 49 & 42& 6 & 84    & \\ \hline
	$G_{\ff_7 \times \ff_7}$ 		& 49 & 36& 12 & 144  & \\ \hline
	$G_{\ff_7 \times \ff_8}$  		& 56 & 42 & 13 & 168 & \\ \hline
	$G_{\ff_7 \times \ff_9}$  		& 63 & 48 & 14 & 192 & \\ \hline
	$G_{\ff_8 \times \ff_8}$  		& 64 & 49 & 14 & 196 & $*$ \\ \hline
	$G_{\ff_8 \times \ff_9}$  		& 72 & 56 & 15 & 224 & \\ \hline
	$G_{\ff_9[x]/(x^2)}$  	    	& 81 & 72 & 8  & 144 & \\ \hline
	$G_{\ff_9 \times \ff_9}$  		& 81 & 64 & 16 & 256 & \\ \hline
	$G_{\Z_{121}}$ 			 		& 121& 110& 10 & 220 & \\ \hline
	$G_{\ff_{11} \times \ff_{11}}$  & 121& 100& 20 & 400 & \\ \hline
    $G_{\Z_{169}}$ 				 	& 169& 157& 11 & 314 & \\ \hline
	\end{tabular}$$
\end{table}

\section{Bigger sets of equienergetic non-isospectral graphs}	
In this final section, by applying the previous results, we will produce bigger sets $\{ \Gamma_1, \ldots, \Gamma_\ell \}$ of equienergetic non-isospectral graphs. For instance, sets of the form $\{ G_{R_1}, G_{R_1}^+, \ldots, G_{R_\ell}, G_{R_\ell}^+\}$
or using the well-known result
\begin{equation} \label{EGH}
E(G_1 \otimes G_2) = E(G_1)E(G_2)
\end{equation}
also sets of the form $\{ G_{R_1} \otimes \Gamma_1, \ldots, G_{R_\ell} \otimes \G_\ell \}$ or 
$\{ G_{R_1} \otimes \Gamma_1, G_{R_1}^+ \otimes \Gamma_1, \ldots, G_{R_\ell} \otimes \Gamma_\ell, G_{R_\ell}^+ \otimes \G_\ell \}$.

We begin by pointing out that it may well happen that $G_R$ is equienergetic to $G_{R'}$ with $R$ local and $R'$ non-local. For instance, $G_{\Z_4}=C_4$ and $G_{\Z_2 \times \Z_2}=K_2 \otimes K_2$ are equienergetic non-isospectral graphs with energy 4 (see Introduction). A similar example can be given using finite fields. We recall that 
$E(G_R) = 2^s |R^*|$, by \eqref{energy GRs}.
If $R=\ff_{q^2}$ and $R'=\ff_q \times \ff_q$ then $E(G_R)=E(G_{R'})$ if and only if $q=3$. 
In this case the energy is given by 
$$E(G_{\ff_9})=2\cdot 8 = 16 = 2^2 \cdot 2 \cdot 2 =E(G_{\ff_3 \times \ff_3}).$$ 
Notice that $G_{\ff_9} = X(\ff_9, \ff_9^*) =K_9$ and that $G_{\ff_3 \times \ff_3}$ is a 4-regular graph with 9 vertices which is also strongly regular graph (Proposition \ref{srg GR}). There is only one such graph (\cite{Shr}), with parameters $srg(9,4,1,2)$ and hence it is the Paley graph $P_2(9) = X(\ff_9 ,\{x^2: x\in \ff_9^*\})$ and also the Hamming graph $H_2(9)$.

By using the equienergetic pair $\{ G_{\ff_9}, G_{\ff_3 \times \ff_3}\}$ 
and Theorem \ref{XRR* equinoiso} 
we now get infinitely many 4-tuples of integral equienergetic non-isospectral graphs.
\begin{prop} \label{F9F3F3-1}
	If $R=R_1\times\cdots\times R_s$ is a commutative ring such that $2m_i < r_i$ for all $i=1,\ldots,s$, then 
	\begin{equation} \label{4tuple}
	\{ G_{\ff_9 \times R}, G_{\ff_9\times R}^+, G_{\ff_3 \times \ff_3 \times R}, G_{\ff_3 \times \ff_3\times R}^+ \}
	\end{equation}
	is a 4-tuple of integral equienergetic non-isospectral connected non-bipartite graphs. Furthermore, all the graphs have or have not loops depending whether $E(R)= \{0\}$ or not, respectively. 
\end{prop}

\begin{proof}
First, notice that $\ff_9\times R$ and $\ff_3\times\ff_3 \times R$ are rings of odd type, 
since they have a non-trivial odd part and $2m_i < r_i$ for all $i=1,\ldots,s$. 
Thus, by Theorem \ref{XRR* equinoiso}, all the graphs in \eqref{4tuple} are integral, connected and non-bipartite; and also $\{G_{\ff_9 \times R}, G_{\ff_9\times R}^+\}$ and $\{G_{\ff_3 \times \ff_3 \times R}, G_{\ff_3 \times \ff_3\times R}^+\}$ are two pairs of equienergetic and non-isospectral graphs. 
		
		Now, $G_{\ff_9 \times R}$ and $G_{\ff_3 \times \ff_3 \times R}$ are equienergetic since, by \eqref{EGH}, we have that 
		\begin{align*}
		& E(G_{\ff_9 \times R}) = E(G_{\ff_9} \otimes G_R)  = E(G_{\ff_9})E(G_R) = 16E(G_R), \\[1mm]
		& E(G_{\ff_3 \times \ff_3 \times R}) = E(G_{\ff_3 \times \ff_3} \otimes G_R) = E(G_{\ff_3 \times \ff_3}) E(G_R)=16E(G_R).
		\end{align*}
		The graphs $G_{\ff_9 \times R}$ and $G_{\ff_3 \times \ff_3 \times R}$ are non-isospectral since they have different regularity degrees. Similarly, one can check that the same occur with the remaining pairs $\{ G_{\ff_9 \times R}, G_{\ff_3 \times \ff_3 \times R}^+\}$ and $\{ G_{\ff_9 \times R}, G_{\ff_3 \times \ff_3 \times R}^+\}$. 
		Therefore, all the graphs in \eqref{4tuple} ar equienergetic and non-isospectral, as desired. 
		The last assertion is a consequence of ($b$) in Lemma \ref{lem nonsym}.
\end{proof}

\begin{rem}
If we try to do the same with the equienergetic non-isospectral pair $\{G_{\Z_4}, G_{\Z_2 \times \Z_2}\}$, we do not  obtain a 4-tuple since, by Lemma \ref{lem nonsym}, the graphs $G_{\Z_4\times R}$ and $G_{\Z_2\times \Z_2\times R}$ are isospectral to $G_{\Z_4\times R}^+ $ and $G_{\Z_2 \times \Z_2\times R}^+$, respectively.
\end{rem}

We now give a similar but different construction which in some sense generalizes the previous proposition. 

\begin{prop} \label{F9F3F3-2}
	If $\G$ is an integral regular connected non-bipartite simple graph then 
	$$\{ G_{\ff_9 } \otimes \G, G_{\ff_9}^+ \otimes \G, G_{\ff_3 \times \ff_3}\otimes \G, G_{\ff_3 \times \ff_3}^+ \otimes \G\}$$ 
	is a 4-tuple of integral equienergetic non-isospectral connected simple regular graphs. 
\end{prop}

\begin{proof}
	All the graphs $G_{\ff_9} \otimes \G$, $G_{\ff_9}^+ \otimes \G$, 
	$G_{\ff_3 \times \ff_3} \otimes \G$, $G_{\ff_3 \times \ff_3}^+ \otimes \G$ are integral, connected, regular and have the same number of vertices $9|V(\G)|$. Since the Kronecker product of a loopless graph with any graph (with or without loops) is loopless, the four graphs in question are simple. 
	To check equienergeticity, note that by \eqref{EGH} we have   
	\begin{align*}
	& E(G_{\ff_9} \otimes \G) = E(G_{\ff_9})E(\G)=16E(\G)= E(G_{\ff_9}^+)E(\G) = E(G_{\ff_9}^+ \otimes \G), \\
	& E(G_{\ff_3 \times \ff_3} \otimes \G)= E(G_{\ff_3 \times \ff_3}) E(\G)=16E(\G) = E(G_{\ff_3 \times \ff_3}^+) E(\G) = E(G_{\ff_3 \times \ff_3}^+ \otimes \G).
	\end{align*}
	
	Finally, we show that all the graphs are mutually non-isospectral. The graphs $G_{\ff_9}\otimes \G$ and $G_{\ff_3 \times \ff_3}\otimes \G$ are non-isospectral since they have different degrees of regularity, and the same happens between $G_{\ff_9}^+\otimes \G$ and $G_{\ff_3 \times \ff_3}^+\otimes \G$. 
	Now, recall that $$\mathrm{Spec}(G_{\ff_9})=\{[8]^1,[-1]^8\} \qquad \text{and} \qquad 
	\mathrm{Spec}(G_{\ff_9}^{+})=\{[8]^1, [1]^4,[-1]^4\}$$ 
	and let $\{[\kappa]^1, [\lambda_{1}]^{\mu_1},\ldots,[\lambda_s]^{\mu_s}\}$ be the spectrum of $\G$,
	where $\kappa$ is the regularity degree of $\G$. 
	Since $\G$ is non-bipartite we have that $|\lambda_i|<\kappa$ for all $i=1,\ldots s$, and hence there exists at most one $i$ such that $-\kappa=8\lambda_i$. Thus, 
	$-\kappa$ is an eigenvalue of $G_{\ff_9}\otimes \G$ with multiplicity $8+\sigma$, where 
$\sigma=\mu_i$ if $-\kappa=8\lambda_i$ for $i=1,\ldots,s$ and 
$\sigma=0$ otherwise, while  $-\kappa$ is an eigenvalue of $G_{\ff_9}^+\otimes \G$ with multiplicity $4+\sigma$. 
	Therefore, $G_{\ff_9}\otimes \G$ and $G_{\ff_9}^+\otimes \G$ are non-isospectral, as asserted. 
	A similar argument shows that 
	$G_{\ff_3\times \ff_3}\otimes \G$ and $G_{\ff_3\times \ff_3}^+\otimes \G$ are non-isospectral as well. 
\end{proof}

\begin{rem}
\noindent ($i$) 	
If $\G$ is bipartite then the graphs $G_R\otimes \G$ and $G_{R}^+\otimes \G$ are isospectral, so the hypothesis `non-bipartite' in the previous proposition cannot be removed.

\noindent ($ii$) 
If $R$ is a ring with $O(R)=\{0\}$, $E(R)\ne \{0\}$ and $2m_i<r_i$ for all $i=1,\ldots,s$ (for instance $R=\ff_4$), then $G_R=G_R^+$ is simple and non-bipartite by Lemma \ref{lem nonsym}. Let $F=\ff_9$ or $F=\ff_3 \times \ff_3$. Then, if we take $\Gamma=G_R$, we have 
$G_{F \times R} = G_F \otimes G_R$ and $G_{F \times R}^+ = G_F^+ \otimes G_R$.
Hence, Proposition~\ref{F9F3F3-2} coincides with Proposition \ref{F9F3F3-1} in this case.
\end{rem}

Now, we produce big sets of equienergetic pairs of graphs. Some of the graphs will be Ramanujan while some others not. 

\begin{exam}
Let $\G_i=G_{R_i}$ for $i=1,2$ where $R_1=\ff_3\times \ff_3$, $R_2=\ff_9$ and let $\G=G_R$ and $\G^+=G_{R}^+$ 
where $R=\ff_4\times \ff_5$.
Since $R$ is an odd type ring, then $\G$ is non-bipartite by Lemma~\ref{lem nonsym}. 
By Propositions \ref{F9F3F3-1} and \ref{F9F3F3-2} we have that 
$$\{ \G_i\otimes\G, \: \G_i^+\otimes \G^+ \}_{i=1,2} \qquad \text{and} \qquad \{ \G_i\otimes\G, \: \G_i^+\otimes \G \}_{i=1,2}$$
are 4-tuples of equienergetic non-isospectral graphs. 
It is not difficult to see that 
$\G_i^+\otimes\G^+$ and $\G_j^+\otimes \G$ are non-isospectral for all $i,j\in \{1,2\}$. 
We can also consider $\G_i\otimes \G^+$ ($i=1,2$). In this case, these graphs are non-isospectral to the above graphs.

Also, since $(4,5)\ne (3,3)$, $\G$ is equienergetic and non-isospectral with $\bar \G$. 
Moreover, we can consider $\G^{-}=X^+(R, R\smallsetminus (R^*\cup\{0\}))$. By Theorem \ref{equienergetic}, this graph is equienergetic with $X(R,R\smallsetminus (R^* \cup\{0\}))=\bar \G$. 
Thus, $\G^{-}$ is equienergetic with $\G$, and it can be seen that $\{\G,\G^+,\bar\G,\G^{-}\}$ is a 4-tuple of non-isospectral graphs.
Thus, we have the following 16-tuple of equienergetic non-isospectral integral graphs of $180$ vertices
$$\{ \G_{i} \otimes \G, \: \G_{i} \otimes \G^+, \: \G_{i} \otimes \bar \G, \: 
 \G_{i} \otimes \G^-, \:  \G_{i}^+ \otimes \G, \: \G_{i}^+ \otimes \G^+, \: \G_{i}^+ \otimes \bar \G, \G_{i}^+ \otimes \G^-\}_{i=1,2}.$$
All the graphs are simple except for the last two, $\G_1^+ \otimes \G^-$ and $\G_2^+ \otimes \G^-$, which have loops. 
\hfill $\lozenge$
\end{exam}

\begin{exam} \label{ej mixed}
Consider the following three pairs of rings 
\begin{align*}
R_1 = \ff_3 \times \ff_4, \qquad R_2 = \ff_5 \times \ff_7, \\ 
R_3 = \ff_3 \times \ff_5, \qquad R_4 = \ff_4 \times \ff_7, \\ 
R_5 = \ff_3 \times \ff_7, \qquad R_6 = \ff_4 \times \ff_5, 
\end{align*}
from Table 1, giving place to the six triples $\{G_{R_i}, G_{R_i}^+, \bar G_{R_i}\}$, $i=1,2,3$, of equienergetic non-isospectral Ramanujan graphs.
Since the rings are products of two different finite fields, for $i=1,2,3$, each pair of rings $G_{R_{2i-1}},G_{R_{2i}}$ 
gives rise to a 9-uple 
\begin{align*}
G_{R_{2i-1}} \otimes G_{R_{2i}}, \qquad G_{R_{2i-1}} \otimes G_{R_{2i}}^+, \qquad G_{R_{2i-1}} \otimes \bar G_{R_{2i}}, \\ 
G_{R_{2i-1}}^+ \otimes G_{R_{2i}}, \qquad G_{R_{2i-1}}^+ \otimes G_{R_{2i}}^+, \qquad G_{R_{2i-1}}^+ \otimes \bar G_{R_{2i}}, \\ 
\bar G_{R_{2i-1}} \otimes G_{R_{2i}}, \qquad \bar G_{R_{2i-1}} \otimes G_{R_{2i}}^+, \qquad \bar G_{R_{2i-1}} \otimes \bar G_{R_{2i}},
\end{align*}
of equienergetic non-isospectral graphs. 
However, note that 
$R_1 \times R_2 \simeq R_3 \times R_4 \simeq R_5 \times R_6 \simeq R$ where $R=\ff_3 \times \ff_4 \times \ff_5 \times \ff_7$ and hence 
$G_{R_1} \otimes G_{R_2} = G_{R_3} \otimes G_{R_4} =G_{R_5} \otimes G_{R_6} = G_R$ 
and $G_{R_1}^+ \otimes G_{R_2}^+ = G_{R_3}^+ \otimes G_{R_4}^+ =G_{R_5}^+ \otimes G_{R_6}^+ = G_R^+$.  
Altogether, without repetitions, there are 23 equienergetic graphs, with $3\cdot 4\cdot 5\cdot 7= 420$ vertices and energy 
$$E(G_R) = 2^4(3-1)(4-1)(5-1)(7-1) = 2^8 \cdot 3^2 = 2304.$$
Although tedious, one can check that these 23 graphs are non-isospectral to each other. 
\hfill $\lozenge$
\end{exam}

\end{document}